\documentclass[final]{siamltex}

\usepackage{amsmath}
\usepackage{mathrsfs}
\usepackage{amsfonts}
\usepackage{amssymb}
\usepackage{graphicx}
\usepackage{psfrag}
\usepackage{placeins}
\usepackage{epsfig}
\usepackage{amsopn}
\usepackage{amstext}
\usepackage{amssymb}
\usepackage{amscd}
\usepackage{latexsym}
\allowdisplaybreaks
\usepackage{geometry}
\usepackage{multirow}
\usepackage{diagbox}
\usepackage{color}

\newtheorem{rem}[theorem]{Remark}

\newcommand{\bbm}{\begin{bmatrix}}
	\newcommand{\ebm}{\end{bmatrix}}

\renewcommand{\theequation}{\arabic{section}.\arabic{equation}}
\usepackage{subfigure}
\usepackage{graphicx}
\usepackage{float}
\usepackage[numbers,sort&compress]{natbib}
\begin{document}

\title{Two-Grid Domain Decomposition Methods  for the Coupled Stokes-Darcy System}
\author{
Yizhong Sun \thanks{%
School of Mathematical Sciences, East China Normal University,
Shanghai, P.R.China. \texttt{bill950204@126.com}.}
\and Feng Shi \thanks{%
College of Science, Harbin Institute of Technology, Shenzhen,
P.R.China. \texttt{shi.feng@hit.edu.cn}.}
\and Haibiao Zheng \thanks{%
School of Mathematical Sciences, East China Normal University,
Shanghai Key Laboratory of Pure Mathematics and Mathematical
Practice, Shanghai, P.R. China. \texttt{hbzheng@math.ecnu.edu.cn}.
Partially supported by   NSF of China (Grant No. 11971174) and NSF
of Shanghai (Grant No. 19ZR1414300) and Science and Technology
Commission of Shanghai Municipality (Grant No. 18dz2271000 and No. 19JC1420102). }
\and Heng Li \thanks{%
School of Mathematical Sciences, East China Normal University,
Shanghai, P.R.China. }
\and Fan Wang \thanks{%
School of Information Science and Engineering, Qufu Normal University,
Rizhao, P.R.China. }}
\maketitle

\begin{abstract}
In this paper, we propose two novel Robin-type domain decomposition methods based on the two-grid techniques for the coupled Stokes-Darcy system. Our schemes firstly adopt the existing Robin-type domain decomposition algorithm to obtain the coarse grid approximate solutions. Then two modified domain decomposition methods are further constructed on the fine grid by utilizing the framework of two-grid methods to enhance computational efficiency, via replacing some interface terms by the coarse grid information. The natural idea of using the two-grid frame to optimize the domain decomposition method inherits the best features of both methods and can overcome some of the domain decomposition deficits. The resulting schemes can be implemented easily using many existing mature solvers or codes in a flexible way, which are much effective under smaller mesh sizes or some realistic physical parameters. Moreover, several error estimates are carried out to show the stability and convergence of the schemes. Finally, three numerical experiments are performed and compared with the classical two-grid method, which verifies the validation and efficiency of the proposed algorithms.
\end{abstract}


\section{Introduction}
Multi-domain, multi-physics coupled problems have arisen in many natural and industrial applications, such as the groundwater fluid flow in the karst aquifer, material in industrial filtration, blood flow motion in the arteries to name just a few. In this paper, we mainly focus on the coupled Stokes-Darcy system, which couples one free fluid flow  (formulated as Stokes equations) with one porous medium flow (described by Darcy's law) by an interface between two subdomains. Many researchers have presented several efficient methods for solving such a coupled system. Especially, decoupling methods allow us to decompose the coupled problems into several independent subproblems, which can be solved simultaneously with some ``legacy" algorithms and codes developed for the Stokes and Darcy equations respectively. We can refer to such decoupling methods as  Lagrange multiplier methods \cite{LMLayton2003, LMGatica03, LMGatica2012}, stabilized mixed finite element methods
\cite{JP2020, Mahbub2019, Mahbub2020}, domain decomposition methods \cite{Discacciati07, Chen, Cao14, Boubendir, Cao11, He15, Vassilev1, Jiang}, two-grid or multi-grid methods \cite{Mu, Cai2009, Zuo14, Hou16,Zuo18}, and optimized Schwarz methods \cite{Discacciati, Gander} and so on.

Among them, domain decomposition methods (DDMs) are much popular ways to deal with large multi-physics problems due it has many off-the-shelf and efficient solvers for each decoupled subproblems \cite{Discacciati07, Chen, Cao14, Boubendir, Cao11, He15, Vassilev1, Jiang}. Based on the characteristics of high precision and easy-to-operation, DDMs have received extensive attention and applications undoubtedly. By introducing the Robin-type interface boundary conditions, DDMs can decouple the multi-domain, multi-physics problems naturally \cite{Discacciati07, Vassilev1}. In particular, Robin-type DDMs were proposed in \cite{Discacciati07} for the Stokes-Darcy coupled problems. Later, Chen et al. \cite{Chen} extended the works of \cite{Discacciati07} and achieved the explicit convergence rates of the order $(1-Ch)$, where $h$ is the mesh size and $C$ is a positive constant. Interestingly, they also obtained $h$-independent convergence result by choosing parameters value appropriately. But the above Robin-type domain decomposition methods have some major drawbacks, which require more iterative steps to solve relatively finer mesh $h$ causes enormous computational cost.  On the other hand, the $h$-independent convergence for the DDMs also requires increasing CPU time to resolve the fine grid solutions for decreasing mesh size $h$.  To overcome these difficulties, we compute the fine grid problems in a decoupling way based on the coarse grid solutions in the framework of two-grid methods. Hence two-types of novel optimized domain decomposition methods are based on the two-grid techniques referred to as ``Two-grid Domain Decomposition methods (TGDDMs)" to improve the computational efficiency significantly.

The basic feature of the proposed TGDDMs algorithms is to apply the parallel Robin-Robin type domain decomposition method to solve iteratively certain steps of the coarse grid problem, and then compute a modified domain decomposition method without any iteration on the fine grid to get the final approximate solution. In this process, some interface terms need to be substituted by the approximate coarse grid solutions. The ideas of utilizing the classical two-grid (CTG) techniques to optimize DDMs not only retain the best features of both methods, but also overcome some of DDMs deficits. Here the coupled Stokes-Darcy problem is naturally decoupled into two independent parallel subproblems. Hence such a physics-based domain decomposition approach of TGDDMs in the coarse grid is more attractive and natural than the CTG methods. On the other hand, the proposed novel TGDDMs can overcome the computational complexity faced by using DDMs for the fine grid approximation which reduce computational significantly.

In this manuscript, firstly, we design a novel \emph{TGDDM1} algorithm, which only requires a few data ($g_{S,H}$ and $g_{D,H}$) from the coarse mesh but may lose minimal accuracy due to some internal factors of interface terms. Then a special variant of \emph{TGDDM1} is proposed, named \emph{TGDDM2}, which is still a modified version of DDMs. The feature of \emph{TGDDM2} requires more data ($\overrightarrow{u_{S,H}}, p_{S,H}, g_{S,H}, g_{D,H}$) from the coarse mesh, with the meaningful profit of improving the error estimates. The most interesting thing is that the error estimates of \emph{TGDDM2} are almost equivalent to the analysis of the CTG methods in \cite{Mu,Hou16}. In \cite{Mu}, Mu et al. were first to get the solutions of the coupled Stokes-Darcy problem on the coarse grid, then two decoupled sub-problems were solved on the fine grid. They achieved $O(h^{\frac{3}{4}})$ for $H=\sqrt{h}$ due to some technical reasons. Later, Hou. \cite{Hou16} achieved optimal error estimate for the sufficiently smooth interface and necessarily introducing an auxiliary function. Note that the idea of \emph{TGDDM2} is different from \cite{Mu,Hou16} since the essence of \emph{TGDDM2} is DDMs only utilizing a two-grid framework, while using DDM on the coarse grid and One-step modified DDM on the fine grid.
We perform three exclusive numerical experiments to validate the proposed methods and illustrate optimal  approximate accuracy, complicated flow characteristics, streamlines, magnitudes, pressure contour, mass conservation, and CPU performance in the pseudo-realistic computational domain for the Stokes-Darcy model. More importantly, we observe that our algorithms are more effective to simulate some coupled problems with smaller $h$ or some realistic physical parameters than the CTG methods of \cite{Mu,Hou16}.

The rest of the paper is organized as follows. In Section 2, we recall the well-known Stokes-Darcy coupled system with the Beavers-Joseph-Saffman interface conditions and related notations. In Section 3, we review the parallel Robin-Robin domain decomposition algorithm and its convergence results.  Section 4 proposes two novel two-grid domain decomposition methods (TGDDMs) and derive their error estimates. In Section 5, several numerical experiments, including 2D and 3D cases, are performed to show the optimization effect and the efficiency of our TGDDMs. Moreover, some comparisons with DDMs and CTG methods are carried out to show the stability and validity of the proposed algorithms. Finally, we summarize our present work and report some future research directions in Section 6.

\section{The Coupled Stokes-Darcy System}

Consider the coupled Stokes-Darcy model on two adjoint bounded domains in $R^{d}(d=2~\mathrm{or}~3)$, denoted by $\Omega _{S}$ for fluid flow, and $\Omega_{D}$ for porous media flow, with an interface $\Gamma$, namely $\Omega _{S}\cap \Omega _{D}={\O }$,  $\overline{\Omega }_{S}\cap\overline{\Omega }_{D}=\Gamma $, and $\overline{\Omega }=\overline{\Omega }_{S}\cup \overline{\Omega }_{D}$. We also refer  $\overrightarrow{n_{S}}$ and $\overrightarrow{n_{D}}$ as the unit outward normal vectors on $\partial \Omega _{S}$ and $\partial \Omega _{D}$, respectively, and  represent  the unit tangential vectors on the interface $\Gamma $  by $\overrightarrow{\tau_{i}},i=1,\cdots ,d-1$. Noting that $\overrightarrow{n_{S}}=-\overrightarrow{n_{D}}~\mathrm{on}~\Gamma $.
We also define by $\Gamma_{S}=\partial \Omega _{S}\setminus \Gamma ,\Gamma _{D}=\partial \Omega_{D}\setminus \Gamma $.  A sketch of the problem domain and its boundaries  is displayed in Fig. \ref{domain2}.

\begin{figure}[htbp]
\centering
\includegraphics[width=90mm,height=60mm]{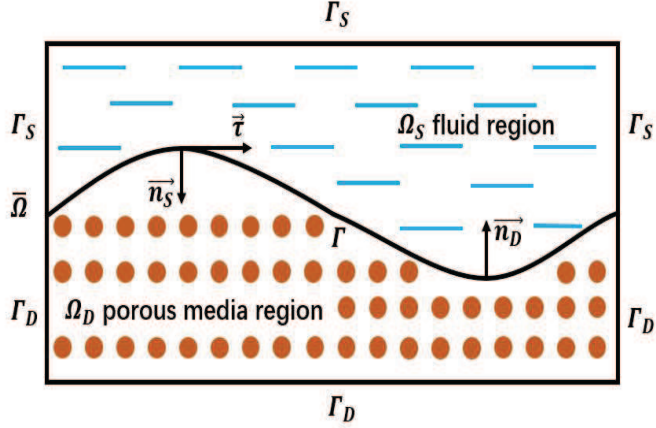}
\caption{The global domain $\overline{\Omega}$ consisting of the fluid region
$\Omega_S$ and the porous media region $\Omega_D$ separated by the
interface $\Gamma$.}\label{domain2}
\end{figure}

In the fluid region $\Omega_{S} $, the fluid velocity $\overrightarrow{u_{S}}$ and kinematic pressure $p_{S}$
are assumed to satisfy the Stokes equations:
\begin{eqnarray}
 -\nabla \cdot \mathbb{T}(\overrightarrow{u_{S}},p_S)&=&\overrightarrow{f_{S}}\ \ \ \ \mathrm{in}~\Omega _{S},
\label{Stokes1} \\
\nabla \cdot \overrightarrow{u_{S}} &=&0\ \ \ \ \hspace{0.8mm} \ \mathrm{in}~\Omega _{S},\label{Stokes2}
\end{eqnarray}
where $\mathbb{T}(\overrightarrow{u_{S}},p_S)=-p_S\mathbb{I}+2\nu\mathbb{D}(\overrightarrow{u_{S}})$ is the stress tensor, with $\mathbb{D}( \overrightarrow{u_{S}})=\frac{1}{2}(\nabla \overrightarrow{u_{S}}+(\nabla \overrightarrow{u_{S}})^{T})$ representing the deformation tensor,  $\mathbb{I}$ denotes the  identity matrix and $\nu $ is the kinematic viscosity of the fluid flow. Besides, $\overrightarrow{f_S}$ expresses the external body force.

In the porous media region $\Omega_{D} $, the flow
is governed by the mixed Darcy equations for the
fluid velocity $\overrightarrow{u_{D}}$ and the
hydraulic head $\phi_{D}$:
\begin{eqnarray}
\overrightarrow{u_{D}}&=&-\mathbb{K}\nabla \phi_{D}\ \ \ \ \hspace{2mm}
\mathrm{in}~\Omega_{D} ,  \label{Darcy1} \\
 \nabla \cdot \overrightarrow{u_{D}} &=&f_{D}\ \ \ \ \hspace{11mm}%
\mathrm{in}~\Omega _{D}, \label{Darcy2}
\end{eqnarray}
where $\mathbb{K} $ is the hydraulic conductivity tensor, and  for simplicity, to be assumed as  $\mathbb{K}=k\mathbb{I}$ with some constant $k$, and the identity matrix $\mathbb{I}$. In addition, $f_{D} $ is a sink/source term, and the piezometric head $\phi_D$ can be described as
$\phi_D=z+\frac{p_D}{\rho g}$, with the height $z$, the dynamic pressure $p_D$,
the density $\rho$, and the gravitational acceleration $g$. By eliminating $\overrightarrow{u_D}$, an equation of the second-order for the hydraulic head $\phi_{D}$ can be obtained:
\begin{eqnarray}
 \nabla \cdot (-\mathbb{K}\nabla \phi_{D}) = f_{D}\ \ \ \ \
\mathrm{in}~\Omega_{D}.  \label{Darcy}
\end{eqnarray}

For the simplicity of presentation, the impermeable boundary on $\partial\Omega\setminus\Gamma$ is assumed, i.e., the fluid velocity $\overrightarrow{u_{S}} $ and the hydraulic head $\phi_{D}$ satisfy homogeneous Dirichlet boundary conditions:
\begin{eqnarray*}
\overrightarrow{u_{S}}=0~~\mathrm{on}~\Gamma _{S}, \hspace{10mm}
\phi_{D}=0~~\mathrm{on}~\Gamma _{D}.
\end{eqnarray*}
It is worth noting that the present schemes and numerical analysis are easily extended to the non-homogeneous case.

On the interface $\Gamma$, the coupling conditions usually consist of  the conservation of mass, the balance
of forces and the tangential conditions related to the normal stress force.
In this paper, we impose the Beavers-Joseph-Saffman (BJS) on the interface $\Gamma $,
see \cite{Beavers, Saffman, Jones}. Therefore, the interface coupling conditions are assumed as:
\begin{eqnarray}
\overrightarrow{u_{S}}\cdot \overrightarrow{n_{S}}&=&\mathbb{K}\nabla \phi_{D}\cdot
\overrightarrow{n_{D}}\ \ \ \ \ \ \ \ \ \hspace{8mm} \mathrm{on}~\Gamma ,  \label{interface1} \\
-\overrightarrow{n_{S}}\cdot (\mathbb{T}(\overrightarrow{u_{S}},p_S)\cdot \overrightarrow{n_{S}})
&=&  g(\phi_{D}-z) \ \ \ \ \ \ \  \ \ \hspace{10mm} \mathrm{on}~\Gamma, \label{interface2}\\
-\overrightarrow{\tau_{i}}\cdot(\mathbb{T}(\overrightarrow{u_{S}},p_S)\cdot \overrightarrow{n_{S}})
&=&\frac{\nu\alpha\sqrt{d}}{\sqrt{\mathrm{trace}(\Pi)}} \overrightarrow{\tau _{i}} \cdot \overrightarrow{u_{S}} \hspace{8mm} 1\leq i \leq d-1 \hspace{8mm} \mathrm{on}~\Gamma,\label{interface3BJS}
\end{eqnarray}%
where $\alpha>0$ represents a dimensionless constant, and $\Pi =\frac{\mathbb{K}\nu}{g} $. For the third condition, we mention here that  more physically realistic Beavers-Joseph (BJ) conditions \cite{CGHW08,Cao10}, are also applicable:
\begin{eqnarray*}
-\overrightarrow{\tau_{i}}\cdot(\mathbb{T}(\overrightarrow{u_{S}},p_S)\cdot \overrightarrow{n_{S}})
=\frac{\nu\alpha\sqrt{d}}{\sqrt{\mathrm{trace}(\Pi)}}
\overrightarrow{\tau _{i}} \cdot (\overrightarrow{u_{S}}
-\overrightarrow{u_{D}} )\hspace{8mm} 1\leq i \leq d-1 \hspace{8mm} \mathrm{on}~\Gamma.
\end{eqnarray*}
However, in practice, the term $\overrightarrow{\tau _{i}} \cdot \overrightarrow{u_{D}}$ is usually neglected compared with $\overrightarrow{\tau _{i}} \cdot \overrightarrow{u_{S}}$, since it is much smaller than the latter.

In the following, we introduce some Sobolev spaces and notations, see for instance  \cite{Sobolev}:
\begin{eqnarray*}
\mathbf{X}_S&:=&\{\overrightarrow{v_S}\in H^1(\Omega_S)^d : \overrightarrow{v_S}=0~ \mathrm{on}~ \Gamma_S\},\\
Q_S&:=&L^2(\Omega_S),\\
{X}_D&:=&\{\phi_D\in H^1(\Omega_D) : \phi_D=0~ \mathrm{on}~ \Gamma_D\}.
\end{eqnarray*}
For the  subdomains $\Omega_S$ and $\Omega_D$, we use $||\cdot||$ and $(\cdot, \cdot)$ to denote the $L^2(\Omega_{S/D})$ norm and inner product, respectively. Moreover, the $H^1(\Omega_{S/D})$ norm is denoted by $||\cdot||_1$. On the interface $\Gamma$, the $L^2(\Gamma)$ inner product is represented as $\langle \cdot,\cdot \rangle$, and the corresponding norm is defined as $||\cdot||_{\Gamma}$.

\section{The Robin-Robin Domain Decomposition Algorithm}
The coupled Stokes-Darcy system is split into two parallel subproblems between the fluid flow $\Omega_S$ region and porous media flow $\Omega_D$ subdomain through well-known domain decomposition method (DDM). In this section, we recall some details of  the Robin-Robin domain decomposition method  proposed by Chen et al. \cite{Chen} for the Stokes-Darcy model. The authors defined two Robin conditions for the Stokes system and the Darcy equation in the following way:

For two given constants  $\delta_{S}, \delta_{D}>0$, there exists two corresponding functions $g_S$ and $g_D$  on the interface $\Gamma$, such that
\begin{eqnarray}
\overrightarrow{n_{S}}\cdot (\mathbb{T}(\overrightarrow{u_{S}},p_S)\cdot \overrightarrow{n_{S}})
+\delta_S \overrightarrow{u_{S}}\cdot \overrightarrow{n_{S}}=g_S,\label{Robintype1} \\
\delta_D \mathbb{K} \nabla \phi_D
\cdot \overrightarrow{n_{D}}+g\phi_D=g_D,\label{Robintype2}
\end{eqnarray}
with the following compatibility conditions:
\begin{eqnarray}
g_S &=& \delta_S \overrightarrow{u_S} \cdot \overrightarrow{n_S}-g\phi_D+gz \hspace{5mm} \mathrm{on}~ \Gamma, \label{comp1}\\
g_D &=& \delta_D \overrightarrow{u_S} \cdot \overrightarrow{n_S}+g\phi_D \hspace{12.2mm} \mathrm{on}~ \Gamma,
\label{comp2}
\end{eqnarray}
Then,  the original interface conditions
(\ref{interface1})-(\ref{interface2}) and the  two newly constructed Robin-type conditions (\ref{Robintype1})-(\ref{Robintype2})  above are equivalent.

The corresponding weak formulation for the decoupled Stokes-Darcy problem with Robin-type boundary conditions (\ref{Robintype1})-(\ref{Robintype2}) can be reformulated as: for
two given functions $g_S,\ g_D\in L^2(\Gamma)$ and two positive constants $\delta_S, \delta_D$, finding $(\overrightarrow{u_{S}},p_S)\in  (\mathbf{X}_{S}, Q_{S})$ and $\phi_D \in \mathbf{X}_D$ such that
\begin{eqnarray}\label{DDMweak}
a_{S}(\overrightarrow{u_{S}},\overrightarrow{v_{S}})-b_{S}(\overrightarrow{v_{S}},{p_S})
+\delta_S\langle\overrightarrow{u_{S}}\cdot \overrightarrow{n_S},
\overrightarrow{v_{S}}\cdot \overrightarrow{n_S}\rangle
&+&\sum_{i=1}^{d-1}\frac{\nu\alpha\sqrt{d}}{\sqrt{\mathrm{trace}(\Pi)}}
\langle\overrightarrow{u_{S}}\cdot \overrightarrow{\tau _{i}},
\overrightarrow{v_{S}}\cdot \overrightarrow{\tau _{i}}\rangle\nonumber\\
\hspace{16mm} &=& \langle{g}_{S},\overrightarrow{v_{S}}\cdot \overrightarrow{n_S}\rangle
+(\overrightarrow{f_{S}},\overrightarrow{v_{S}})
\hspace{10mm}  \forall ~\overrightarrow{v_{S}}\in \mathbf{X}_{S}, \\
\label{DDMweak-2}\ \ \ \ \  \ \ \ \ \ \ \ \
b_{S}(\overrightarrow{u_{S}},q)&=&0 \hspace{42.7mm} \forall ~q\in Q_{S},\\
\label{DDMweak-3}\ \
\delta_D a_{D}(\phi_D,\psi) + \langle g \phi_D,\psi\rangle
&=& \langle{g}_{D}, \psi\rangle+ \delta_D (f_{D},\psi)
\hspace{14.4mm} \forall ~\psi \in \mathbf{X}_{D},
\end{eqnarray}
with the   compatibility conditions (\ref{comp1})-(\ref{comp2}). To this end, some bilinear forms are introduced:
\begin{eqnarray*}
a_{S}(\overrightarrow{u_{S}},\overrightarrow{v_{S}})
&:=& 2\nu (\mathbb{D}(\overrightarrow{u_{S}}),\mathbb{D}(\overrightarrow{v_{S}})), \\
b_{S}(\overrightarrow{v_{S}},q)&:=& (\nabla \cdot \overrightarrow{v_{S}},q),\\
a_{D}(\phi_D,\psi)&:=& (\mathbb{K}\nabla \phi_D,\nabla \psi).
\end{eqnarray*}
Chen et al. proved the well-posedness of above weak formulation  (\ref{DDMweak})-(\ref{DDMweak-3}) in  \cite{Chen}.

The following finite element discretization needs to be considered for the Robin-Robin domain decomposition methods. Let $\mathcal{T}_{h}$ be a regular and quasi-uniform triangulation of $ \overline{\Omega}$ with the mesh size $h>0$ , which consists of two consistent triangulations $\mathcal{T}_{S,h}$ and $\mathcal{T}_{D,h}$ for the subdomains $ \Omega_{S} $ and $ \Omega_{D} $, respectively. Moreover, we denote by $\mathcal{B}_{h}$ the partition of $\Gamma$ induced from $\mathcal{T}_{h}$.

Then we define the finite element spaces $ \mathbf{X}_{S,h} \subset \mathbf{X}_{S}, Q_{S,h}\subset Q_{S}$ and  $ {X}_{D,h} \subset {X}_{D}$ for approximating the fluid velocity and pressure in the free flow region, and the piezometric head in the porous medium flow, respectively. Actually, one popular choice is applying the \emph{P2-P1} finite elements for Stokes problem, and matching \emph{P2} finite elements for Darcy problem \cite{Chen}. In details, the finite element spaces used here can be defined by:
\begin{eqnarray*}
\mathbf{X}_{S,h}&:=&\{\overrightarrow{{v}_{S,h}}\in (H^{1}(\Omega _{S}))^{d}: \ \ %
\overrightarrow{{v}_{S,h}}|_{T}\in (\mathbb{P}_{2}(T))^{d} \ \ %
\forall ~T\in \mathcal{T}_{S,h}, \ \ %
\overrightarrow{{v}_{S,h}}|_{\Gamma_{S} }= {0} \},\\
{Q}_{S,h}&:=&\{{q}_{S,h}\in L^{2}(\Omega _{S}): \ \ \ \ \ \ \ %
{q}_{S,h}|_{T}\in \mathbb{P}_{1}(T)\ \ \ \hspace{2.5mm} \forall ~T\in \mathcal{T}_{S,h} \},\\
{X}_{D,h}&:=&\{\psi_{D,h}\in H^{1}(\Omega _{D}): \ \ \ \ \hspace{0.5mm} %
\psi_{D,h}|_{T}\in \mathbb{P}_{2}(T)\ \ \hspace{1.5mm} \ \forall ~T\in \mathcal{T}_{D,h}, \ \ %
\psi_{D,h}|_{\Gamma_{D}}={0} \},
\end{eqnarray*}
The spaces $\mathbf{X}_{S,h} $ and ${Q}_{S,h}$ satisfy the
discrete LBB or inf-sup condition \cite{proveLBB,proveLBB2}.

Specifically, the discrete trace space on the interface $\Gamma$ is defined as follows:
\begin{eqnarray*}
{Z}_{h} ~:= ~\{g_{h}\in L^{2}(\Gamma ):\ \ %
g_{h}|_{\tau }\in \mathbb{P}_{2}(\tau )\ \ \forall ~\tau \in \mathcal{B}_{h}, \ \
g_{h}|_{\partial \Gamma}= {0} \}
= ~\mathbf{X}_{S,h}|_{\Gamma} \cdot \overrightarrow{n_{S}}
= ~{X}_{D,h}|_{\Gamma}.
\end{eqnarray*}

Based on the above finite element discretization and the weak formulation (\ref{DDMweak})-(\ref{DDMweak-3}), Chen. et al. \cite{Chen} presented a parallel finite element Robin-Robin DDM for solving the coupled Stokes-Darcy problem.

\noindent{\textbf{FE DDM Algorithm (\emph{DDM-Chen}).}}

1. Initial values of $g_{S,h}^0 \in Z_h$ and $g_{D,h}^0 \in Z_h$ are guessed, which can be chosen as zeros.

2. For $n=0,1,2,\cdots, N$, the Stokes and Darcy systems can be solved independently with Robin-type boundary conditions (\ref{Robintype1})-(\ref{Robintype2}): find
 $(\overrightarrow{u_{S,h}^n}, p_{S,h}^n) \in (\mathbf{X}_{S,h}, Q_{S,h})$ and $\phi_{D,h}^n\in{X}_{D,h}$ satisfies
\begin{eqnarray}\label{decoupled}
a_{S}(\overrightarrow{u_{S,h}^n},\overrightarrow{v_{S}})-b_{S}(\overrightarrow{v_{S}},{p_{S,h}^n})
&+&\delta_S\langle\overrightarrow{u_{S,h}^n}\cdot \overrightarrow{n_S},
\overrightarrow{v_{S}}\cdot \overrightarrow{n_S}\rangle
+\sum_{i=1}^{d-1}\frac{\nu\alpha\sqrt{d}}{\sqrt{\mathrm{trace}(\Pi)}}
\langle\overrightarrow{u_{S,h}^n}\cdot \overrightarrow{\tau _{i}},
\overrightarrow{v_{S}}\cdot \overrightarrow{\tau _{i}}\rangle\nonumber\\
\hspace{15mm} &=& \langle{g}_{S,h}^n,\overrightarrow{v_{S}}\cdot \overrightarrow{n_S}\rangle
+(\overrightarrow{f_{S}},\overrightarrow{v_{S}})
\hspace{26.5mm}  \forall ~\overrightarrow{v_{S}}\in \mathbf{X}_{S,h}, \\
\label{decoupled-2}\ \ \ \ \  \ \ \ \ \ \ \ \ \ \ \ \ \
b_{S}(\overrightarrow{u_{S,h}^n},q)&=&0 \hspace{62mm} \forall ~q\in Q_{S,h},
\end{eqnarray}
and
\begin{eqnarray}
\label{decoupled-3}\ \
\delta_D a_{D}(\phi_{D,h}^n,\psi) + \langle g \phi_{D,h}^n,\psi\rangle
= \langle {g}_{D,h}^n, \psi \rangle+ \delta_D (f_{D},\psi)
\hspace{13.5mm} \forall ~\psi \in {X}_{D,h},
\end{eqnarray}
respectively.

3. Update $g_{S,h}^{n+1},\ g_{D,h}^{n+1} \in Z_h$ for further iteration in the following manner:
\begin{eqnarray*}
&&  {g}^{n+1}_{S,h}=({\delta_S}/{\delta_D}){g}^{n}_{D,h}
-(1+{\delta_S}/{\delta_D})g \phi_{D,h}^n+gz,\\
&&  {g}^{n+1}_{D,h}=-{g}^{n}_{S,h}+(\delta_S+\delta_D)\overrightarrow{u_{S,h}^{n}} \cdot \overrightarrow{n_S}+gz.
\end{eqnarray*}

In step 2, the way of choosing $N$ will be described in the numerical section. The convergence analysis of \emph{DDM-Chen}  derived  in \cite{Chen}. For the case $\delta_S=\delta_D=\delta$, \emph{DDM-Chen} proved the convergence rate of $1-Ch$. On the other hand, authors achieved $h$-independent error estimate for $\delta_S<\delta_D$ under the suitable choice of the parameters $\delta_S$ and $\delta_D$ including some control conditions. Furthermore, the Robin-Robin domain decomposition method leads to a finite element approximation for the decoupled Stokes-Darcy problem with Robin-type boundary conditions (\ref{Robintype1})-(\ref{Robintype2}):  for
two given functions $g_{S,h},\ g_{D,h}\in Z_h$ and two positive constants $\delta_S, \delta_D$, searching $(\overrightarrow{u_{S,h}},p_{S,h})\in  (\mathbf{X}_{S,h}, Q_{S,h})$ and $\phi_{D,h} \in {X}_{D,h}$ such that
\begin{eqnarray}\label{FEDDMweak}
a_{S}(\overrightarrow{u_{S,h}},\overrightarrow{v_{S}})-b_{S}(\overrightarrow{v_{S}},{p_{S,h}})
+&\delta_S&\langle\overrightarrow{u_{S,h}}\cdot \overrightarrow{n_S},
\overrightarrow{v_{S}}\cdot \overrightarrow{n_S}\rangle
+\sum_{i=1}^{d-1}\frac{\nu\alpha\sqrt{d}}{\sqrt{\mathrm{trace}(\Pi)}}
\langle\overrightarrow{u_{S,h}}\cdot \overrightarrow{\tau _{i}},
\overrightarrow{v_{S}}\cdot \overrightarrow{\tau _{i}}\rangle\nonumber\\
\hspace{15mm} &=& \langle{g}_{S,h},\overrightarrow{v_{S}}\cdot \overrightarrow{n_S}\rangle
+(\overrightarrow{f_{S}},\overrightarrow{v_{S}})
\hspace{20.5mm}  \forall ~\overrightarrow{v_{S}}\in \mathbf{X}_{S,h}, \\
\label{FEDDMweak-2}\ \ \ \ \  \ \ \ \ \ \ \ \ \ \ \ \ \
b_{S}(\overrightarrow{u_{S,h}},q)&=&0 \hspace{55.7mm} \forall ~q\in Q_{S,h},\\
\label{FEDDMweak-3}\ \
\delta_D a_{D}(\phi_{D,h},\psi) + \langle g \phi_{D,h},\psi\rangle
&=& \langle{g}_{D,h}, \psi\rangle+ \delta_D (f_{D},\psi)
\hspace{24.7mm} \forall ~\psi \in {X}_{D,h},
\end{eqnarray}
with the FE compatibility conditions:
\begin{eqnarray}
g_{S,h} &=& \delta_S \hspace{0.5mm} \overrightarrow{u_{S,h}} \cdot \overrightarrow{n_S}-g \hspace{0.5mm} \phi_{D,h}+gz \hspace{5.5mm} \mathrm{on}~ \Gamma, \label{FEcomp1}\\
g_{D,h} &=& \delta_D \hspace{0.5mm} \overrightarrow{u_{S,h}} \cdot \overrightarrow{n_S}+g \hspace{0.5mm} \phi_{D,h} \hspace{13mm} \mathrm{on}~ \Gamma.
\label{FEcomp2}
\end{eqnarray}

\section{The Two-grid Domain Decomposition Algorithms}
Although the Robin-Robin domain decomposition methods are well-studied and Chen et al. \cite{Chen} achieved convergence rate generally $(1-Ch)$ for $\delta_S=\delta_D=\delta$, but it causes enormous computational cost due to increases of iterative steps for the finer mesh scales $ h $. Even $h$-independent convergence for the case of $\delta_S<\delta_D$ requires massive computational effort for the fine grid solutions with decreasing mesh size  $h$. The computational accuracy, efficiency, and effectiveness can be improved by combining two benchmark methods together: one is domain decomposition methods, and the other is classical two-grid methods. To this end, two-types of novel optimized DDMs based on the two-grid techniques are designed and analyzed from the aspect of saving the CPU time and computational resources, referred to as ``Two-grid Domain Decomposition methods (TGDDMs)" hereafter.

\subsection{TGDDM Algorithm 1 (\emph{TGDDM1})}
The first two-grid domain decomposition algorithm for solving the coupled Stokes-Darcy problem require the following two steps:
\begin{enumerate}
\item Recall \emph{DDM-Chen} to solve the problem (\ref{FEDDMweak})-(\ref{FEDDMweak-3}) on a coarse grid with mesh size $H$. We can finally obtain coarse grid data  $g_{S,H}, g_{D,H} \in Z_H$.
\item
A modified fine grid problem is constructed and solved by
finding $(\overrightarrow{u_{S}^h}, p_{S}^h) \in (\mathbf{X}_{S,h}, Q_{S,h})$, $\phi_{D}^h\in {X}_{D,h}$, such that
\begin{eqnarray}\label{FineDDMweak}
a_{S}(\overrightarrow{u_{S}^h},\overrightarrow{v_{S}})-b_{S}(\overrightarrow{v_{S}},{p_{S}^h})
+&\delta_S&\langle\overrightarrow{u_{S}^h}\cdot \overrightarrow{n_S},
\overrightarrow{v_{S}}\cdot \overrightarrow{n_S}\rangle
+\sum_{i=1}^{d-1}\frac{\nu\alpha\sqrt{d}}{\sqrt{\mathrm{trace}(\Pi)}}
\langle\overrightarrow{u_{S}^h}\cdot \overrightarrow{\tau _{i}},
\overrightarrow{v_{S}}\cdot \overrightarrow{\tau _{i}}\rangle\nonumber\\
&=& \langle{g}_{S,H},\overrightarrow{v_{S}}\cdot \overrightarrow{n_S}\rangle
+(\overrightarrow{f_{S}},\overrightarrow{v_{S}})
\hspace{14.5mm}  \forall ~\overrightarrow{v_{S}}\in \mathbf{X}_{S,h}, \ \ \ \ \ \\
\label{FineDDMweak-2}\ \ \ \ \  \ \ \ \ \ \ \ \ \ \ \ \ \
b_{S}(\overrightarrow{u_{S}^h},q)&=&0 \hspace{50.5mm} \forall ~q\in Q_{S,h},\ \ \ \ \ \\
\label{FineDDMweak-3}\ \
\delta_D a_{D}(\phi_{D}^h,\psi) + \langle g \phi_{D}^h,\psi\rangle
&=& \langle{g}_{D,H}, \psi\rangle+ \delta_D (f_{D},\psi)
\hspace{18.7mm} \forall ~\psi \in {X}_{D,h}.
\end{eqnarray}
\end{enumerate}

 The well-posedness of the finite element approximation for the modified fine grid problem (\ref{FineDDMweak})-(\ref{FineDDMweak-3}) can be easily checked. More importantly, the newly constructed Algorithm \emph{TGDDM1} can admit us to decouple the Stokes-Darcy system into two types of subproblems, namely the Stokes one and the Darcy one, consequently, we can utilize some ``legacy" codes to simulate both subproblems in parallel by further applying  mutually studied non-overlapping or particularly overlapping DDMs for each subproblem.

\subsection{Error Analysis of  \emph{TGDDM1}}
In this subsection, we derive the error estimate of the algorithm \emph{TGDDM1}. Before that we  introduce some preliminaries and basic results for convenience.  We  use the notation $x \lesssim y$ as $x \leq Cy$, where the generic constant $C$ may denote different values in different places.
Assume that the solutions $\overrightarrow{u_S}, p_S$ to (\ref{Stokes1})-(\ref{Stokes2}) and $\phi_D$ to (\ref{Darcy}) admit the regularities of $\overrightarrow{u_S} \in H^3(\Omega_S)^d, p_S \in H^2(\Omega_S)$ and $\phi_D \in H^3(\Omega_D)$. Recall that the error estimates for the decoupled scheme in \cite{Chen}
for the Stokes-Darcy problem with Robin-type boundary conditions  hold:
\begin{eqnarray*}
&&||\overrightarrow{u_S}-\overrightarrow{u_{S,h}}||_1 \lesssim \hspace{0.5mm}h^2, \hspace{9.8mm} ||\overrightarrow{u_S}-\overrightarrow{u_{S,h}}|| \lesssim \hspace{0.5mm}h^3, \hspace{9.5mm}
||p_S-p_{S,h}|| \lesssim \hspace{0.5mm}h^2,\\
&&||\phi_D-\phi_{D,h}||_1 \lesssim \hspace{0.5mm}h^2, \hspace{8.5mm}
||\phi_D-\phi_{D,h}|| \lesssim \hspace{0.5mm}h^3.
\end{eqnarray*}

For the finite element approximation (\ref{FEDDMweak})-(\ref{FEDDMweak-3}), some error functions, related with the differences of the solution components on the coarse grid and fine grid, are denoted by
\begin{eqnarray*}
&&\overrightarrow{e_{S,H}}=\overrightarrow{u_{S,h}}-\overrightarrow{u_{S,H}},\hspace{5.7mm}
{\varepsilon}_{S,H}=p_{S,h}-{p}_{S,H}, \hspace{5mm}
{\eta}_{S,H}=g_{S,h}-{g}_{S,H},\\
&&{e}_{D,H}={\phi}_{D,h}-{\phi}_{D,H},\hspace{3.4mm}
{\eta}_{D,H}=g_{D,h}-{g}_{D,H}.
\end{eqnarray*}
Then, several basic but important error estimates regarding with the numerical solutions to (\ref{FEDDMweak})-(\ref{FEDDMweak-3}) on the coarse and fine grids can be easily derived by using  the triangle inequality, for instance:
\begin{eqnarray*}
&&||\overrightarrow{e_{S,H}}||_1 \lesssim \hspace{0.5mm}H^2, \hspace{9.8mm} ||\overrightarrow{e_{S,H}}|| \lesssim \hspace{0.5mm}H^3, \hspace{9.5mm} ||{\varepsilon}_{S,H}|| \lesssim \hspace{0.5mm}H^2,\\
&&||{e_{D,H}}||_1 \lesssim \hspace{0.5mm}H^2, \hspace{9mm} ||{e_{D,H}}|| \lesssim \hspace{0.5mm}H^3.
\end{eqnarray*}

In Lemma \ref{CFerr-comp}, we derive the estimates for
 the FE compatibility conditions (\ref{FEcomp1})-(\ref{FEcomp2}) for the coarse and fine grids by utilizing the above error estimates above for (\ref{FEDDMweak})-(\ref{FEDDMweak-3}).  Note that these estimates will play the vital role in the numerical analysis hereafter.

\begin{lemma}\label{CFerr-comp}
Along the interface $\Gamma$, two error estimates on $\eta_{S,H}$ and $\eta_{D,H}$ related with the interface conditions (\ref{FEcomp1})-(\ref{FEcomp2}) have the forms of
\begin{eqnarray*}
&&||{\eta_{S,H}}||_{\Gamma} \lesssim (\delta_S+g)\hspace{0.5mm}H^{\frac{5}{2}},
\hspace{10mm} ||{\eta_{D,H}}||_{\Gamma} \lesssim (\delta_D+g)\hspace{0.5mm}H^{\frac{5}{2}}.
\end{eqnarray*}
\end{lemma}

\begin{proof}
According to the FE compatibility conditions (\ref{FEcomp1})-(\ref{FEcomp2}), we can directly derive the following relations:
\begin{eqnarray*}
\eta_{S,H} = \delta_S \hspace{0.5mm} \overrightarrow{e_{S,H}} \cdot \overrightarrow{n_S}-g\hspace{0.5mm} e_{D,H}, \hspace{5mm}
\eta_{D,H} = \delta_D \hspace{0.5mm} \overrightarrow{e_{S,H}} \cdot \overrightarrow{n_S}+g\hspace{0.5mm} e_{D,H}.
\end{eqnarray*}
By the virtue of Young's inequality, trace inequality (see (2.10) in \cite{trace}) with a positive constant $C_{tr}$ and the basic error estimates above, we can deduce that
\begin{eqnarray*}
||\eta_{S,H}||_{\Gamma} &=& ||\delta_S \hspace{0.5mm} \overrightarrow{e_{S,H}} \cdot \overrightarrow{n_S}-g\hspace{0.5mm} e_{D,H}||_{\Gamma}\\
&\leq& \delta_S ||\overrightarrow{e_{S,H}} \cdot \overrightarrow{n_S}||_{\Gamma}+ g||e_{D,H}||_{\Gamma}\\
&\leq& \delta_S C_{tr} ||\overrightarrow{e_{S,H}}||^{\frac{1}{2}} ||\overrightarrow{e_{S,H}}||_1^{\frac{1}{2}}
+ g C_{tr} ||e_{D,H}||^{\frac{1}{2}} ||e_{D,H}||_1^{\frac{1}{2}}\\
&\lesssim&  (\delta_S+g) \hspace{0.5mm} H^{\frac{5}{2}}.
\end{eqnarray*}
Similarly, the error estimate of $||\eta_{D,H}||_{\Gamma}$ can be demonstrated.
\end{proof}

Consequently, we have the error analysis for Algorithm \emph{TGDDM1} in the following theorem.
\begin{theorem}\label{Error}
Let $(\overrightarrow{u_{S,h}}, p_{S,h}; \phi_{D,h})$ be the solution of (\ref{FEDDMweak})-(\ref{FEDDMweak-3}), and assume that $(\overrightarrow{u_{S}^h}, p_{S}^h; \phi_{D}^h)$ is the solution of (\ref{FineDDMweak})-(\ref{FineDDMweak-3}) derived from \emph{TGDDM1}, the following error estimates hold:
\begin{eqnarray}
&&||\phi_{D,h}-\phi_{D}^h||_1 \lesssim \frac{\delta_D+g}{k\delta_D} \hspace{0.5mm} H^{\frac{5}{2}}, \label{errDarcy}\\
&&||\overrightarrow{u_{S,h}}-\overrightarrow{u_S^h}||_1 \lesssim (\delta_S+g)\hspace{0.5mm} H^{\frac{5}{2}}, \label{errStokes1}\\
&&||p_{S,h}-p_S^h|| \lesssim \Big(2\nu+\delta_S C_{tr}^2+\frac{\nu\alpha\sqrt{d}C_{tr}^2}{\sqrt{\mathrm{trace}(\Pi)}}+C_{tr} \Big) (\delta_S+g) \hspace{0.5mm} H^{\frac{5}{2}}. \label{errStokes2}
\end{eqnarray}
\end{theorem}

\begin{proof}
On the fine grid, subtracting (\ref{FineDDMweak})-(\ref{FineDDMweak-3}) from (\ref{FEDDMweak})-(\ref{FEDDMweak-3}) yields
\begin{eqnarray}
\label{errFE-1}\ \
\delta_D a_{D}(\phi_{D,h}-\phi_D^h,\psi) + \langle g (\phi_{D,h}-\phi_D^h),\psi\rangle
&=& \langle{g}_{D,h}-g_{D,H}, \psi\rangle
\hspace{15mm} \forall ~\psi \in {X}_{D,h},\\
\label{errFE-2} \ \
a_{S}(\overrightarrow{u_{S,h}}-\overrightarrow{u_S^h},\overrightarrow{v_{S}})
-b_{S}(\overrightarrow{v_{S}},{p_{S,h}}-p_S^h)
&+&\delta_S\langle(\overrightarrow{u_{S,h}}-\overrightarrow{u_S^h})\cdot \overrightarrow{n_S},
\overrightarrow{v_{S}}\cdot \overrightarrow{n_S}\rangle\nonumber\\
\hspace{5mm} &+&\sum_{i=1}^{d-1}\frac{\nu\alpha\sqrt{d}}{\sqrt{\mathrm{trace}(\Pi)}}
\langle(\overrightarrow{u_{S,h}}-\overrightarrow{u_S^h})\cdot \overrightarrow{\tau _{i}},
\overrightarrow{v_{S}}\cdot \overrightarrow{\tau _{i}}\rangle\nonumber\\
\hspace{15mm} &=& \langle{g}_{S,h}-g_{S,H},\overrightarrow{v_{S}}\cdot \overrightarrow{n_S}\rangle
\hspace{8mm} \forall ~\overrightarrow{v_S}\in \mathbf{X}_{S,h},\\
\label{errFE-3}\ \
b_{S}(\overrightarrow{u_{S,h}}-\overrightarrow{u_S^h},q)&=&0 \hspace{38.5mm} \forall ~q\in Q_{S,h}.
\end{eqnarray}
For the Darcy system, choosing $\psi=\phi_{D,h}-\phi_{D}^h \in {X}_{D,h}$ in (\ref{errFE-1}) gives
\begin{eqnarray}
\delta_D a_{D}(\phi_{D,h}-\phi_D^h,\phi_{D,h}-\phi_D^h) + g ||\phi_{D,h}-\phi_D^h||_{\Gamma}^2
= \langle{g}_{D,h}-g_{D,H}, \phi_{D,h}-\phi_D^h\rangle. \label{errD}
\end{eqnarray}
Thanks to Cauchy-Schwarz inequality, trace inequality, and Lemma \ref{CFerr-comp}, the error result for the porous media region can be deduced from (\ref{errD}) as
\begin{eqnarray*}
||\phi_{D,h}-\phi_D^h||_1^2 &\leq& \frac{1}{k}a_{D}(\phi_{D,h}-\phi_D^h,\phi_{D,h}-\phi_D^h)\\
&\leq& \frac{1}{k\delta_D}\Big[ \delta_D a_{D}(\phi_{D,h}-\phi_D^h,\phi_{D,h}-\phi_D^h) + g ||\phi_{D,h}-\phi_D^h||_{\Gamma}^2 \Big]\\
&\leq& \frac{1}{k\delta_D} \langle{g}_{D,h}-g_{D,H}, \phi_{D,h}-\phi_D^h\rangle\\
&\leq& \frac{1}{k\delta_D}||\eta_{D,H}||_{\Gamma}||\phi_{D,h}-\phi_D^h||_{\Gamma}
 \lesssim \frac{\delta_D+g}{k\delta_D} \hspace{0.5mm} H^{\frac{5}{2}} ||\phi_{D,h}-\phi_D^h||_1,
\end{eqnarray*}
which immediately yields the error inequality (\ref{errDarcy}).

For the Stokes problem, taking $\overrightarrow{v}=\overrightarrow{u_{S,h}}-\overrightarrow{u_S^h} \in \mathbf{X}_{S,h}$ in (\ref{errFE-2}) and $q=p_{S,h}-p_{S}^h \in Q_{S,h}$ in (\ref{errFE-3}), then adding the resulting equations together, we can obtain
\begin{eqnarray*}
a_{S}(\overrightarrow{u_{S,h}}-\overrightarrow{u_S^h},\overrightarrow{u_{S,h}}-\overrightarrow{u_S^h})
&+&\delta_S ||(\overrightarrow{u_{S,h}}-\overrightarrow{u_S^h})\cdot \overrightarrow{n_S}||_{\Gamma}
+\sum_{i=1}^{d-1}\frac{\nu\alpha\sqrt{d}}{\sqrt{\mathrm{trace}(\Pi)}}
||(\overrightarrow{u_{S,h}}-\overrightarrow{u_S^h})\cdot \overrightarrow{\tau _{i}}||_{\Gamma}\\
\hspace{5mm}&=& \langle{g}_{S,h}-g_{S,H},(\overrightarrow{u_{S,h}}-\overrightarrow{u_S^h})\cdot \overrightarrow{n_S}\rangle.
\end{eqnarray*}
Due to Korn's inequality, namely, there exist a positive constant $C_{K}$, such that
\begin{eqnarray*}
||\overrightarrow{v_S}||_1^2 \leq \frac{C_K^2}{2\nu} a_{S}(\overrightarrow{v_{S}},\overrightarrow{v_{S}})\quad \overrightarrow{v_S} \in \mathbf{X}_S,
\end{eqnarray*}
we can then get
\begin{eqnarray*}
||\overrightarrow{u_{S,h}}-\overrightarrow{u_S^h}||_1^2&\leq& \frac{C_K^2}{2\nu} \Big[a_{S}(\overrightarrow{u_{S,h}}-\overrightarrow{u_S^h},\overrightarrow{u_{S,h}}-\overrightarrow{u_S^h})
+\delta_S ||(\overrightarrow{u_{S,h}}-\overrightarrow{u_S^h})\cdot \overrightarrow{n_S}||_{\Gamma}\\
&& \hspace{20mm}+\sum_{i=1}^{d-1}\frac{\nu\alpha\sqrt{d}}{\sqrt{\mathrm{trace}(\Pi)}}
||(\overrightarrow{u_{S,h}}-\overrightarrow{u_S^h})\cdot \overrightarrow{\tau _{i}}||_{\Gamma}\Big]\\
&\leq& \frac{C_K^2}{2\nu} \langle{g}_{S,h}-g_{S,H},(\overrightarrow{u_{S,h}}-\overrightarrow{u_S^h})\cdot \overrightarrow{n_S}\rangle\\
&\leq& \frac{C_K^2}{2\nu}||\eta_{S,H}||_{\Gamma}||\overrightarrow{u_{S,h}}-\overrightarrow{u_S^h}||_{\Gamma}
\lesssim (\delta_S+g) \hspace{0.5mm} H^{\frac{5}{2}} ||\overrightarrow{u_{S,h}}-\overrightarrow{u_S^h}||_1.
\end{eqnarray*}
The proof for the error estimate (\ref{errStokes1}) is completed.

To check the last estimate (\ref{errStokes2}), we need to use the discrete Brezzi-Babuska condition on the free fluid domain $\Omega_S$, namely, for $q=p_{S,h}-p_{S}^h \in Q_{S,h}$, there exists $\overrightarrow{v_S} \in \mathbf{X}_{S,h}$, such that
\begin{eqnarray*}
||p_{S,h}-p_{S}^h|| &\leq& \frac{-b_{S}(\overrightarrow{v_{S}},{p_{S,h}}-p_S^h)}{||\overrightarrow{v_S}||_1}.
\end{eqnarray*}
Then applying the equation (\ref{errFE-2}) by such $\overrightarrow{v_S}$,  we can arrive at
\begin{eqnarray*}
||p_{S,h}-p_{S}^h|| &\leq& \frac{1}{||\overrightarrow{v_S}||_1} \Big[|a_{S}(\overrightarrow{u_{S,h}}-\overrightarrow{u_S^h},\overrightarrow{v_{S}})|
+\delta_S|<(\overrightarrow{u_{S,h}}-\overrightarrow{u_S^h})\cdot \overrightarrow{n_S},
\overrightarrow{v_{S}}\cdot \overrightarrow{n_S}>|\\
&&\hspace{2mm}+\sum_{i=1}^{d-1}\frac{\nu\alpha\sqrt{d}}{\sqrt{\mathrm{trace}(\Pi)}}
|\langle(\overrightarrow{u_{S,h}}-\overrightarrow{u_S^h})\cdot \overrightarrow{\tau _{i}},
\overrightarrow{v_{S}}\cdot \overrightarrow{\tau _{i}}\rangle|
+|\langle{g}_{S,h}-g_{S,H},\overrightarrow{v_{S}}\cdot \overrightarrow{n_S}\rangle|\Big]\\
&\leq& \Big(2\nu+\delta_S C_{tr}^2+\frac{\nu\alpha\sqrt{d}C_{tr}^2}{\sqrt{\mathrm{trace}(\Pi)}} \Big) ||\overrightarrow{u_{S,h}}-\overrightarrow{u_S^h}||_1 + C_{tr} ||{g}_{S,h}-g_{S,H}||_{\Gamma}\\
&\lesssim& \Big(2\nu+\delta_S C_{tr}^2+\frac{\nu\alpha\sqrt{d}C_{tr}^2}{\sqrt{\mathrm{trace}(\Pi)}}+C_{tr} \Big) (\delta_S+g) \hspace{0.5mm} H^{\frac{5}{2}},
\end{eqnarray*}
which completes the proof of (\ref{errStokes2}).
\end{proof}

Based on the theorem above, by further utilizing the triangle inequalities, we can immediately derive the error estimates of the solution of \emph{TGDDM1} compared with the exact one in the following.
\begin{corollary}\label{errorder}
Let $(\overrightarrow{u_{S}^h}, p_{S}^h; \phi_{D}^h) \in (\mathbf{X}_{S,h}, Q_{S,h}; {X}_{D,h})$
and $(\overrightarrow{u_{S}}, p_{S}; \phi_{D}) \in (\mathbf{X}_{S}, Q_{S}; {X}_{D})$ be the solutions
to \emph{TGDDM1} and  (\ref{DDMweak})-(\ref{DDMweak-3})  respectively,
for the coupled Stokes-Darcy model. For the choice of $H=h^{2/3}$, we have the error estimates
\begin{eqnarray*}
||\phi_{D}-\phi_{D}^h||_1 \lesssim \hspace{0.5mm} h^{\frac{5}{3}},\hspace{5mm}
||\overrightarrow{u_{S}}-\overrightarrow{u_S^h}||_1
+||p_{S}-p_S^h|| \lesssim \hspace{0.5mm} h^{\frac{5}{3}}.
\end{eqnarray*}
Moreover, if specially selecting $H=h^{4/5}$, the error estimates are improved as:
\begin{eqnarray}\label{errTGDDM1}
||\phi_{D}-\phi_{D}^h||_1 \lesssim \hspace{0.5mm} h^2,\hspace{5mm}
||\overrightarrow{u_{S}}-\overrightarrow{u_S^h}||_1
+||p_{S}-p_S^h|| \lesssim \hspace{0.5mm} h^2.
\end{eqnarray}
\end{corollary}

\begin{rem}\label{remark}
Since the error estimates (\ref{errDarcy})-(\ref{errStokes2}) is not optimal for $H=h^{2/3}$ by directly applying the trace inequality, consequently the estimate (\ref{errTGDDM1}) is not optimal. We only utilize the coarse mesh data  $g_{S,H}, g_{D,H}$ and the trace inequality directly, which is shown in Lemma \ref{CFerr-comp}. Hence, some internal factors of interface terms affect our analysis. Based on some numerical experiments performed in the next section, it is recommended to have the optimal result in the form of (\ref{errTGDDM1}) for $H=h^{2/3}$. This optimal estimate can be resolved by some rigorous analysis in further study. In the next subsection, we have presented   improved version of error estimates with the minor modifications of some additional interface terms.
\end{rem}

\subsection{TGDDM Algorithm 2 (\emph{TGDDM2})}
Based on the discussion in Remark \ref{remark}, we improve Algorithm TGDDM1 by utilizing more solution information of \emph{DDM-Chen} on the coarse grid. The main strategy is to  modify both the Robin-type interface terms $\langle{g}_{S,H},\overrightarrow{v_{S}}\cdot \overrightarrow{n_S}\rangle, \langle{g}_{D,H}, \psi\rangle$ in the fine grid, and the interface terms $\delta_S\langle\overrightarrow{u_{S,H}}\cdot \overrightarrow{n_S},\overrightarrow{v_{S}}\cdot \overrightarrow{n_S}\rangle$ and $g\langle \phi_{D,H},\psi\rangle$, simultaneously. Note that this algorithm is utilizing DDM on the coarse grid and One-step modified DDM on the fine grid in essence. We refer the improved algorithm as  \emph{TGDDM2}, whose main idea is summarized as follows:
\begin{enumerate}
\item Solve a coarse grid problem (\ref{FEDDMweak})-(\ref{FEDDMweak-3}) with mesh size $H$ by recalling \emph{DDM-Chen}. Then, we can get the data: $\overrightarrow{u_{S,H}}\in \mathbf{X}_{S,H}$, $\phi_{D,H}\in {X}_{D,H}$ and $g_{S,H}, g_{D,H} \in Z_H$.
\item
The new fine grid problem is constructed by solving:
 $(\overrightarrow{u_{S}^h}, p_{S}^h) \in (\mathbf{X}_{S,h}, Q_{S,h})$, $\phi_{D}^h\in {X}_{D,h}$,  such that
for any $\overrightarrow{v_{S}}\in \mathbf{X}_{S,h}$, $q\in Q_{S,h}$ and $\psi \in {X}_{D,h}$
\begin{eqnarray}\label{2FineDDMweak}
a_{S}(\overrightarrow{u_{S}^h},\overrightarrow{v_{S}})&-&b_{S}(\overrightarrow{v_{S}},{p_{S}^h})
+\sum_{i=1}^{d-1}\frac{\nu\alpha\sqrt{d}}{\sqrt{\mathrm{trace}(\Pi)}}
\langle\overrightarrow{u_{S}^h}\cdot \overrightarrow{\tau _{i}},
\overrightarrow{v_{S}}\cdot \overrightarrow{\tau _{i}}\rangle\nonumber\\
&=&-\delta_S\langle\overrightarrow{u_{S,H}}\cdot \overrightarrow{n_S},
\overrightarrow{v_{S}}\cdot \overrightarrow{n_S}\rangle
+\langle{g}_{S,H},\overrightarrow{v_{S}}\cdot \overrightarrow{n_S}\rangle
+(\overrightarrow{f_{S}},\overrightarrow{v_{S}}),\hspace{8.5mm}\\
\label{2FineDDMweak-2}\ \ \ \ \  \ \ \ \ \ \ \ \ \ \ \ \ \
b_{S}(\overrightarrow{u_{S}^h},q)&=&0, \\
\label{2FineDDMweak-3}\ \
\delta_D a_{D}(\phi_{D}^h,\psi)
&=&- g\langle \phi_{D,H},\psi\rangle + \langle{g}_{D,H}, \psi\rangle+ \delta_D (f_{D},\psi).
\end{eqnarray}
\end{enumerate}

According to the first step in the Algorithm \emph{TGDDM2}  and the FE compatibility conditions (\ref{FEcomp1})-(\ref{FEcomp2}), we can directly show the following relations for further analysis:
\begin{eqnarray}
g_{S,H} &=& \delta_S \hspace{0.5mm} \overrightarrow{u_{S,H}} \cdot \overrightarrow{n_S}-g \hspace{0.5mm} \phi_{D,H}+gz,\label{gSH}\\
g_{D,H} &=& \delta_D \hspace{0.5mm} \overrightarrow{u_{S,H}} \cdot \overrightarrow{n_S}+g \hspace{0.5mm} \phi_{D,H}.\label{gDH}
\end{eqnarray}

Then, by subtracting (\ref{2FineDDMweak})-(\ref{2FineDDMweak-3}) from (\ref{FEDDMweak})-(\ref{FEDDMweak-3}) on the fine grid, and taking account of the conditions (\ref{gSH})-(\ref{gDH}), we can get the  the error functions
\begin{eqnarray}
\label{2errFE-1}\ \
a_{D}(\phi_{D,h}-\phi_D^h,\psi)&=&
\langle(\overrightarrow{u_{S,h}}-\overrightarrow{u_{S,H}})\cdot\overrightarrow{n_S}, \psi\rangle
\hspace{17.5mm} \forall ~\psi \in {X}_{D,h},\\
\label{2errFE-2} \ \
a_{S}(\overrightarrow{u_{S,h}}-\overrightarrow{u_S^h},\overrightarrow{v_{S}})
-b_{S}(\overrightarrow{v_{S}},{p_{S,h}}-&p_S^h&)
+\sum_{i=1}^{d-1}\frac{\nu\alpha\sqrt{d}}{\sqrt{\mathrm{trace}(\Pi)}}
\langle(\overrightarrow{u_{S,h}}-\overrightarrow{u_S^h})\cdot \overrightarrow{\tau _{i}},
\overrightarrow{v_{S}}\cdot \overrightarrow{\tau _{i}}\rangle\nonumber\\
\hspace{2mm} &=& - g\langle\phi_{D,h}-\phi_{D,H},\overrightarrow{v_{S}}\cdot \overrightarrow{n_S}\rangle
\hspace{3mm} \hspace{9.4mm} \forall ~\overrightarrow{v_S}\in \mathbf{X}_{S,h},\\
\label{2errFE-3}\ \
b_{S}(\overrightarrow{u_{S,h}}-\overrightarrow{u_S^h},q)&=&0 \hspace{49.5mm} \forall ~q\in Q_{S,h}.
\end{eqnarray}

It is interesting to discover the newly constructed scheme \emph{TGDDM2} seems to be related with the parameters $\delta_S, \delta_D$ in the second step. Only imposing the Dirichlet traces of the coarse solutions in the fine grid problems leads to the cancellation of the terms related to these two parameters in the residual equations above, hence, the present \emph{TGDDM2} only depends on the parameters in the first step for solving the coarse mesh problems. 

The further error analysis based on equations (\ref{2errFE-1})-(\ref{2errFE-3}) can be easily derived by the similar way with \cite{Mu}. Actually, we can select $\psi = \phi_{D,h}-\phi_D^h \in X_{D,h}$ in (\ref{2errFE-1}), and introduce an auxiliary system:
\begin{eqnarray*}
-\triangle \gamma &=& 0 \hspace{22mm} \mathrm{in} \ \Omega_S,\\
\gamma &=& \phi_{D,h}-\phi_D^h \hspace{8mm} \mathrm{on} \ \Gamma,\\
\gamma &=& 0 \hspace{22mm} \mathrm{on}\ \Gamma_S,
\end{eqnarray*}
where $\gamma \in H^1(\Omega_S)$ is the harmonic extension of $\phi_{D,h}-\phi_D^h$ to the Stokes region. Then, following the analysis in \cite{Mu}, we can get
$||\phi_{D,h}-\phi_{D}^h||_1 \lesssim \hspace{0.5mm} H^{3}$,
and taking $\overrightarrow{v_{S}} = \overrightarrow{u_{S,h}}-\overrightarrow{u_S^h} \in \mathbf{X}_{S,h}$ in (\ref{2errFE-2}) and $q = p_{S,h}-p_S^h \in Q_{S,h}$ in (\ref{2errFE-3}), we can directly obtain
$||\overrightarrow{u_{S,h}}-\overrightarrow{u_S^h}||_1 \lesssim \hspace{0.5mm} H^{\frac{5}{2}}$.
Then, utilizing the inf-sup or LBB condition on $\Omega_S$, we can yield
$||p_{S,h}-p_S^h|| \lesssim \hspace{0.5mm} H^{\frac{5}{2}}$
Further utilizing the triangle inequality, we can finally derive the error estimates under the relation of mesh scale $H=h^{2/3}$:
\begin{eqnarray}\label{result}
||\phi_{D}-\phi_{D}^h||_1 \leq C \hspace{0.5mm} h^{2},\hspace{5mm}
||\overrightarrow{u_{S}}-\overrightarrow{u_S^h}||_1
+||p_{S}-p_S^h|| \leq C \hspace{0.5mm} h^{\frac{5}{3}},
\end{eqnarray}
where $(\overrightarrow{u_{S}^h}, p_{S}^h; \phi_{D}^h) \in (\mathbf{X}_{S,h}, Q_{S,h}; {X}_{D,h})$ and $(\overrightarrow{u_{S}}, p_{S}; \phi_{D}) \in (\mathbf{X}_{S}, Q_{S}; {X}_{D})$ are solutions to
\emph{TGDDM2} and the original coupled Stokes-Darcy problem, respectively. Noting that  the estimate $||\phi_{D}-\phi_{D}^h||_1$ for the Darcy hydraulic head  is improved to the optimal convergence result of $O(h^2)$ here. This will be verified by several numerical experiments in the next section.

If we want to further improve the error estimate in the free fluid region, with a $1/3$ higher order than the result in (\ref{result}),  more rigorous estimates should be used, see for instance  \cite{Hou16}, in which Hou assumed the sufficient smoothness of the interface $\Gamma$ and introduced an auxiliary problem to achieve the proof of the optimal estimate. This result  is also valid for our scheme \emph{TGDDM2}. To this end, we reach the optimal error estimates for \emph{TGDDM2}, $i.e.$,
\begin{eqnarray}\label{opterr}
||\phi_{D}-\phi_{D}^h||_1 \leq C \hspace{0.5mm} h^{2},\hspace{5mm}
||\overrightarrow{u_{S}}-\overrightarrow{u_S^h}||_1
+||p_{S}-p_S^h|| \leq C \hspace{0.5mm} h^{2}.
\end{eqnarray}

\begin{rem}\label{remark2}
We can extend the error analysis for the TGDDMs above to other finite element spaces, such as the well-known MINI elements ($P1b-P1$) for the Stokes region, and the matched $P1$ elements for the Darcy model. Under such choices of the finite element spaces, the discrete trace space ${Z}_{h}$ on the interface $\Gamma$ should be modified as the space of the $P1$ elements, however, the following relations are still true:
\begin{eqnarray*}
{Z}_{h}= ~\mathbf{X}_{S,h}|_{\Gamma} \cdot \overrightarrow{n_{S}}= ~{X}_{D,h}|_{\Gamma}.
\end{eqnarray*}
Following the idea of \cite{Chen}, the scheme of \emph{DDM-Chen} with $P1b-P1-P1$ elements will  admit similar converge results up to minor modification of numerical analysis. Consequently, we can also demonstrate the error estimates between the exact  and finite element solutions in the form of
\begin{eqnarray*}
||\phi_{D}-\phi_{D}^h||_1 \leq C \hspace{0.5mm} h,\hspace{5mm}
||\overrightarrow{u_{S}}-\overrightarrow{u_S^h}||_1
+||p_{S}-p_S^h|| \leq C \hspace{0.5mm} h.
\end{eqnarray*}
\end{rem}

\section{Numerical Experiments}

In this section, we conduct several numerical experiments to illustrate the accuracy and efficiency of the proposed TGDDMs algorithms by simulating the coupled Stokes-Darcy fluid flow model.  In the first numerical test, we achieve the optimal convergence accuracy for the newly proposed \emph{TGDDM1} or \emph{TGDMM2} algorithms for an analytical solution. The second example demonstrates the flow speed, streamlines, conservation of mass, and CPU performance in a conceptual domain. Moreover, we simulate the coupling of the ``3D Shallow Water" system with the Darcy equation to present the complicated flow characteristics.  The finite element spaces are constructed by well-known \emph{P2-P1} elements for the Stokes problem and the matched \emph{P2} elements for the Darcy problem. All the numerical tests are executed by the open-source software FreeFEM++ \cite{F18}.

The stopping criteria for the iterative process of \emph{DDM-Chen} or the first step of the proposed TGDDMs, referred to as the maximal number $N$ of iterations, is usually selected as a fixed tolerance of $10^{-6}$  between two successive iterative solutions of the velocities and piezometric heads in the sense of $L^2$-norm, $i.e.$,
\begin{eqnarray*}
\Bigl(||\overrightarrow{u^{n+1}_{S,H}}-\overrightarrow{u_{S,H}^n}||^2+||\phi_{D,H}^{n+1}-\phi_{D,H}^{n}||^2 \Bigr)^{1/2}\leq 10^{-6}.
\end{eqnarray*}

\subsection{Experiment 1}
The first testing example with exact solution is adapted from \cite{Chen}  to verify the predicted error estimates of ``Remark \ref{remark}", and  check the feasibility of \emph{TGDDM2}. The global domain   $\overline{\Omega}$ consists of the free fluid flow region $\Omega_{S}= [0,\pi]\times[0,1]$ and the porous medium region $\Omega_{D}=[0,\pi]\times[-1,0]$, including the the interface  $\Gamma= \{ 0 \leq x \leq \pi, y=0\}$.
The exact solution is selected as:
\begin{eqnarray*}
\overrightarrow{u_S}=[v'(y) \cos x, v(y) \sin x]^{T}, \hspace{5mm} p_S=0,
\hspace{5mm} \phi_D=(e^{y}-e^{-y}) \sin x,
\end{eqnarray*}
where $v(y)$ denotes any function satisfying the boundary conditions $v(0)=-2k$ and $v'(0)=0$, respectively. Here we select $v(y)=-2k+\frac{k}{\pi^2} \sin^2 (\pi y)$. On the other hand, the above exact solution of such functions  satisfy the Beavers-Joseph-Saffman interface conditions. For the computational convenience, we assume $z=0$, and  other physical parameters $ \nu, g, k$ and $\alpha$ to be $ 1.0 $.

\begin{table}[!h]
\caption{The convergence performances  for the proposed \emph{TGDDM1}, \emph{TGDDM2} algorithms  with different $\delta_S$ and $\delta_D$ by using $ P2-P1-P2 $ finite elements triple.} \label{errTGDDM-E1}\tabcolsep
0pt \vspace*{-10pt}
\par
\begin{center}
\def\temptablewidth{1.0\textwidth}
{\rule{\temptablewidth}{1pt}}
\begin{tabular*}{\temptablewidth}{@{\extracolsep{\fill}}cc||ccc||c|c|c||c|c|c}
\hline
\multirow{2}*{$\delta_S$} & \multirow{2}*{$\delta_D$} & \multirow{2}*{$ H $} & \multirow{2}*{$N$} & \multirow{2}*{$ h $} & \multicolumn{3}{c||}{\emph{TGDDM1}} & \multicolumn{3}{c}{\emph{TGDDM2}}\\
\cline{6-11}
& & & & & $\frac{||\overrightarrow{u_S}-\overrightarrow{u_{S}^h}||_{1}}
{||\overrightarrow{u_S}||_{1}}$ \hspace{1mm} &$\frac{||p_S-p_{S}^h||}{||p_S||}$ \hspace{2mm} & $\frac{||\phi_D-\phi_D^h||_1}{||\phi_D||_1}$ \hspace{1mm} & $\frac{||\overrightarrow{u_S}-\overrightarrow{u_{S}^h}||_{1}}
{||\overrightarrow{u_S}||_{1}}$ \hspace{1mm} & $\frac{||p_S-p_{S}^h||}{||p_S||}$ \hspace{2mm} & $\frac{||\phi_D-\phi_D^h||_1}{||\phi_D||_1}$ \\
\hline\hline
2  &          2  & $\frac{1}{4} $  & 74  & $\frac{1}{8}  $ & 0.2226120 & 0.0386038 & 0.0310740 & 0.2226370 & 0.0382719 & 0.0310642 \\
   &             & $\frac{1}{9} $  & 110 & $\frac{1}{27} $ & 0.0265787 & 0.0028930 & 0.0033393 & 0.0265605 & 0.0027112 & 0.0033308 \\
   &             & $\frac{1}{16}$  & 136 & $\frac{1}{64} $ & 0.0051409 & 0.0004530 & 0.0006324 & 0.0051391 & 0.0004311 & 0.0006319 \\
   &             & $\frac{1}{25}$  & 154 & $\frac{1}{125}$ & 0.0013372 & 0.0001078 & 0.0001634 & 0.0013368 & 0.0001015 & 0.0001634 \\
   &             & $\frac{1}{36}$  & 168 & $\frac{1}{216}$ & 0.0004507 & 0.0000375 & 0.0000573 & 0.0004504 & 0.0000352 & 0.0000572 \\

\hline
\multicolumn{5}{c||}{Rate} & 1.99 & 1.93 & 1.92 & 1.99 & 1.94 & 1.92 \\
 \hline\hline
1  &          1  & $\frac{1}{4} $  & 64  & $\frac{1}{8}  $ & 0.2226080 & 0.0383720 & 0.0310927 & 0.2226370 & 0.0382710 & 0.0310639 \\
   &             & $\frac{1}{9} $  & 92  & $\frac{1}{27} $ & 0.0265639 & 0.0027403 & 0.0033477 & 0.0265606 & 0.0027115 & 0.0033303 \\
   &             & $\frac{1}{16}$  & 101 & $\frac{1}{64} $ & 0.0051394 & 0.0004345 & 0.0006329 & 0.0051390 & 0.0004296 & 0.0006319 \\
   &             & $\frac{1}{25}$  & 104 & $\frac{1}{125}$ & 0.0013364 & 0.0000952 & 0.0001631 & 0.0013362 & 0.0000928 & 0.0001630 \\
   &             & $\frac{1}{36}$  & 104 & $\frac{1}{216}$ & 0.0004478 & 0.0000327 & 0.0000550 & 0.0004465 & 0.0000325 & 0.0000549 \\
\hline
\multicolumn{5}{c||}{Rate} & 2.00 & 1.95 & 1.99 & 2.00 & 1.92 & 1.99 \\
\hline\hline
$\frac{1}{2}$& 1 & $\frac{1}{4} $  & 17  & $\frac{1}{8}  $ & 0.2226171 & 0.0383043 & 0.0310931 & 0.2226370 & 0.0382714 & 0.0310645 \\
   &             & $\frac{1}{9} $  & 17  & $\frac{1}{27} $ & 0.0265603 & 0.0027045 & 0.0033485 & 0.0265606 & 0.0027112 & 0.0033309 \\
   &             & $\frac{1}{16}$  & 17  & $\frac{1}{64} $ & 0.0051391 & 0.0004303 & 0.0006328 & 0.0051390 & 0.0004295 & 0.0006317 \\
   &             & $\frac{1}{25}$  & 17  & $\frac{1}{125}$ & 0.0013359 & 0.0000871 & 0.0001629 & 0.0013359 & 0.0000866 & 0.0001628 \\
   &             & $\frac{1}{36}$  & 17  & $\frac{1}{216}$ & 0.0004462 & 0.0000305 & 0.0000544 & 0.0004462 & 0.0000303 & 0.0000544 \\
\hline
\multicolumn{5}{c||}{Rate} & 2.00 & 1.92 & 2.01 & 2.00 & 1.92 & 2.00 \\
\hline\hline
$\frac{1}{3}$& 1 & $\frac{1}{4} $  & 21  & $\frac{1}{8}  $ & 0.2226231 & 0.0382893 & 0.0310931 & 0.2226370 & 0.0382714 & 0.0310645 \\
   &             & $\frac{1}{9} $  & 21  & $\frac{1}{27} $ & 0.0265601 & 0.0027016 & 0.0033485 & 0.0265606 & 0.0027113 & 0.0033309 \\
   &             & $\frac{1}{16}$  & 21  & $\frac{1}{64} $ & 0.0051390 & 0.0004296 & 0.0006328 & 0.0051390 & 0.0004295 & 0.0006317 \\
   &             & $\frac{1}{25}$  & 21  & $\frac{1}{125}$ & 0.0013359 & 0.0000868 & 0.0001629 & 0.0013359 & 0.0000866 & 0.0001628 \\
   &             & $\frac{1}{36}$  & 21  & $\frac{1}{216}$ & 0.0004455 & 0.0000303 & 0.0000544 & 0.0004462 & 0.0000304 & 0.0000543 \\
\hline
\multicolumn{5}{c||}{Rate} & 2.01 & 1.92 & 2.01 & 2.00 & 1.92 & 2.01 \\
\hline\hline
\end{tabular*}%
\end{center}
\end{table}
Firstly, we check the convergence of TGDDMs, by solving the coupled Stokes-Darcy problem with the proposed \emph{TGDDM1} or \emph{TGDMM2} on the uniform triangular meshes with size $H=\{1/4,1/9,1/16,1/25,1/36\}$ and $h=H^{3/2}=\{1/8,1/27,1/64,1/125,1/216\}$. The approximate errors are displayed in Table \ref{errTGDDM-E1}, for the velocity of free fluid flow in $H^1$-norm, the pressures of Stokes in $L^2$-norm, and the hydraulic heads of Darcy in $H^1$-norm. As expected, all the approximate errors achieve the optimal orders of $O(h^2)$ for the proposed TGDDMs algorithms. Moreover, the approximate accuracy hold the Remarks  \ref{remark} and  \ref{remark2} and obtain the similar convergence order as Chen et al. \cite{Chen}. On the other hand,   DDM iterative step $N$ in the process of computation on each coarse grid indicate that  the convergence rate is $(1-Ch)$ while $\delta_S=\delta_D$ and  $h$-independent while $\delta_S<\delta_D$. Besides, optimal error orders  are confirmed when $\delta_S < \delta_D$ for $H=h^{2/3}$, supports the Remark \ref{remark} for \emph{TGDDM1} and the result (\ref{opterr}) for \emph{TGDDM2}.



Secondly, we record the CPU time to simulate the proposed TGDDMs algorithms and \emph{DDM-Chen} in Table \ref{CPUtime-E1}   for the comparison purpose. As expected, the proposed \emph{TGDDM1} and  \emph{TGDDM2} algorithms consume less CPU time than the \emph{DDM-Chen}. Besides, the  iterative numbers $ N $ for the three schemes   \emph{DDM-Chen}, \emph{TGDDM1} and \emph{TGDDM2} verify the theoretical results.   The above numerical results illustrate the computational efficiency and exclusive features of our methods.

\begin{table}[!h]
\caption{The DDM iterative step and CPU performer of \emph{DDM-Chen}, \emph{TGDDM1} and \emph{TGDDM2}  by using $ P2-P1-P2 $ finite elements triple.} \label{CPUtime-E1}\tabcolsep
0pt \vspace*{-10pt}
\par
\begin{center}
\def\temptablewidth{1\textwidth}
{\rule{\temptablewidth}{0.9pt}}
\begin{tabular*}{\temptablewidth}{@{\extracolsep{\fill}}ccc|cc|cc|cc}
\hline
\multirow{2}*{$\delta_S$}& \multirow{2}*{$\delta_D$}  & \multirow{2}*{$ h $ \hspace{2mm}} &\multicolumn{2}{c|}{\emph{DDM-Chen}} & \multicolumn{2}{c|}{\emph{TGDDM1}} & \multicolumn{2}{c}{\emph{TGDDM2}}\\
\cline{4-9}
& & & $N$ & CPU\hspace{2mm} & $N$ & CPU & $N$ & CPU \\
\hline
1& 1 & $\frac{1}{8}  $ \hspace{2mm} & 86  &  4.233\hspace{2mm}    & 64   & 1.890  & 64  & 2.193  \\
 &   & $\frac{1}{27} $ \hspace{2mm} & 104 &  37.654\hspace{2mm}   & 92   & 6.773  & 92  & 8.307  \\
 &   & $\frac{1}{64} $ \hspace{2mm} & 105 &  204.799\hspace{2mm}  & 101  & 23.137 & 101 & 22.101 \\
 &   & $\frac{1}{125}$ \hspace{2mm} & 106 &  902.407\hspace{2mm} & 104  & 53.426 & 104 & 53.858 \\
 &   & $\frac{1}{216}$ \hspace{2mm} &  -- &     --    & 104  & 131.924\hspace{2mm} & 104 & 130.292 \\
\hline
$\frac{1}{3}$& 1 & $\frac{1}{8}  $ \hspace{2mm} & 21 &  0.971\hspace{2mm}   & 21  & 0.563  & 21 & 0.550 \\
   &             & $\frac{1}{27} $ \hspace{2mm} & 21 &  7.986\hspace{2mm}   & 21  & 1.579  & 21 & 1.702 \\
   &             & $\frac{1}{64} $ \hspace{2mm} & 21 &  42.980\hspace{2mm}  & 21  & 4.860  & 21 & 4.677 \\
   &             & $\frac{1}{125}$ \hspace{2mm} & 21 &  159.014\hspace{2mm} & 21  & 13.781 & 21 & 12.856\\
   &             & $\frac{1}{216}$ \hspace{2mm} & -- &    --\hspace{2mm}    & 21  & 39.571 & 21 & 37.093 \\
 \hline
\end{tabular*}%
\end{center}
\end{table}

\begin{table}[!h]
\caption{Approximate accuracy and CPU time in the first and second steps for \emph{TGDDM2} and \emph{CTG} by using $P2-P1-P2$ finite elements triple.} \label{ctg}\tabcolsep
0pt \vspace*{-10pt}
\par
\begin{center}
\def\temptablewidth{1.0\textwidth}
{\rule{\temptablewidth}{0.95pt}}
\begin{tabular*}{\temptablewidth}{@{\extracolsep{\fill}}c|ccccc|cccc}
\hline
\multirow{2}*{Algorithms} & \multirow{2}*{$ H $} & \multicolumn{4}{c|}{Step 1 (coarse mesh)} & \multicolumn{4}{c}{Step 2 (fine mesh $h=H^{3/2}$)}\\
\cline{3-10}
 & & $\frac{||\overrightarrow{u_S}-\overrightarrow{u_{S,H}}||_{1}}
{||\overrightarrow{u_S}||_{1}}$ \hspace{1mm} &$\frac{||p_S-p_{S,H}||}{||p_S||}$ \hspace{2mm} & $\frac{||\phi_D-\phi_{D,H}||_1}{||\phi_D||_1}$ \hspace{1mm} & CPU & $\frac{||\overrightarrow{u_S}-\overrightarrow{u_{S}^h}||_{1}}
{||\overrightarrow{u_S}||_{1}}$ \hspace{1mm} & $\frac{||p_S-p_{S}^h||}{||p_S||}$ \hspace{2mm} & $\frac{||\phi_D-\phi_D^h||_1}{||\phi_D||_1}$ & CPU \\
\hline\hline
  & $\frac{1}{16}$ & 0.0799951 & 0.0086573 & 0.0094345 & 2.347\hspace{0.5mm} &\hspace{0.5mm} 0.0051390\hspace{1.5mm} & 0.0004295\hspace{1.5mm} & 0.0006317\hspace{1.5mm} & 1.803 \\
  & $\frac{1}{25}$ & 0.0332764 & 0.0027174 & 0.0040821 & 5.298\hspace{0.5mm} &\hspace{0.5mm} 0.0013359\hspace{1.5mm} & 0.0000866\hspace{1.5mm} & 0.0001628\hspace{1.5mm} & 7.146 \\
  \emph{TGDDM2} & $\frac{1}{36}$ & 0.0157737 & 0.0014256 & 0.0018632 & 10.59\hspace{0.5mm} &\hspace{0.5mm} 0.0004462\hspace{1.5mm} & 0.0000324\hspace{1.5mm} & 0.0000544\hspace{1.5mm} & 22.43 \\
  ($\delta_D=1.0$ & $\frac{1}{49}$ & 0.0081499 & 0.0005610 & 0.0009248 & 21.35\hspace{0.5mm} &\hspace{0.5mm} 0.0001788\hspace{1.5mm} & 0.0000113\hspace{1.5mm} & 0.0000216\hspace{1.5mm} & 61.05 \\
  $\delta_S=\frac{1}{2}\delta_D$) & $\frac{1}{64}$ & 0.0051390 & 0.0004286 & 0.0006210 & 32.65\hspace{0.5mm} & --\hspace{1.5mm} & --\hspace{1.5mm} & 0.0000097\hspace{1.5mm} & -- \\
  & $\frac{1}{81}$ & 0.0031685 & 0.0002209 & 0.0003767 & 53.80\hspace{0.5mm} & --\hspace{1.5mm} & --\hspace{1.5mm} & 0.0000049\hspace{1.5mm} & -- \\
  & $\frac{1}{100}$ &0.0020913 & 0.0001448 & 0.0002483 & 83.51\hspace{0.5mm} & --\hspace{1.5mm} & --\hspace{1.5mm} & --\hspace{1.5mm} & -- \\
\hline
Rate & & 1.97 & 2.00 & 1.98 &  & 1.98 & 2.28 & 1.93 & \\
\hline\hline
  & $\frac{1}{16}$ & 0.0865251 & 0.0074120 & 0.0095117 & 0.736\hspace{0.5mm} &\hspace{0.5mm} 0.0051390\hspace{1.5mm} & 0.0004297\hspace{1.5mm} & 0.0006323\hspace{1.5mm} & 1.151 \\
  & $\frac{1}{25}$ & 0.0332054 & 0.0031460 & 0.0039961 & 1.992 \hspace{0.5mm} &\hspace{0.5mm} 0.0013359\hspace{1.5mm} & 0.0000865\hspace{1.5mm} & 0.0001634\hspace{1.5mm} & 4.616 \\
  & $\frac{1}{36}$ & 0.0174951 & 0.0033676 & 0.0020405 & 5.605\hspace{0.5mm} &\hspace{0.5mm} 0.0004451\hspace{1.5mm} & 0.0000953\hspace{1.5mm} & 0.0001257\hspace{1.5mm} & 17.39 \\
  \emph{CTG} & $\frac{1}{49}$ & 0.0091535 & 0.0042908 & 0.0010872 & 16.65\hspace{0.5mm} &\hspace{0.5mm} 0.0001792\hspace{1.5mm} & 0.0001435\hspace{1.5mm} & 0.0001107\hspace{1.5mm} & 63.22 \\
  & $\frac{1}{64}$ & 0.0058313 & 0.0055461 & 0.0006438 & 37.53\hspace{0.5mm} & --\hspace{1.5mm} & --\hspace{1.5mm} & 0.0000859\hspace{1.5mm} & -- \\
  & $\frac{1}{81}$ & 0.0046892 & 0.0070227 & 0.0005392 & 301.2\hspace{0.5mm} & --\hspace{1.5mm} & --\hspace{1.5mm} & 0.0002123\hspace{1.5mm} & -- \\
  & $\frac{1}{100}$ &0.0049351 & 0.0093033 & 0.0008602 & 465.7\hspace{0.5mm} & --\hspace{1.5mm} & --\hspace{1.5mm} & --\hspace{1.5mm} & -- \\
\hline
\multirow{2}*{Rate} & & 2.10$_{(\frac{49}{36})}$ & 1.92$_{(\frac{25}{16})}$ & 1.96$_{(\frac{64}{49})}$ & & \multirow{2}*{1.97} & 2.39$_{(\frac{125}{64})}$ & 2.02$_{(\frac{125}{64})}$ & \\
     & & -0.24 & -1.34 & -2.22 & & & -0.89 & -2.56 & \\
\hline\hline
\end{tabular*}%
\end{center}
\end{table}
In  Table \ref{ctg}, we  compare the approximate errors and CPU performance between the algorithms \emph{TGDDM2} and \emph{CTG} by running the first step computation  on the coarse mesh and the second step on the fine mesh, respectively. For this purpose, we apply \emph{GMRES} solver to compute the coupled system on the coarse mesh for \emph{CTG}. Here the notation `--' means that the implementation of our present computer is out of memory. The results indicate that \emph{TGDDM} and \emph{CTG} both have optimal convergence orders and much similar approximate accuracy when the fine mesh size is not too small. Table \ref{ctg} illustrate that the Stokes problem is out of memory if we consider  $h\leq \frac{1}{512}$ (corresponding to $h=H^{3/2}$, $H\leq\frac{1}{64}$ ). On the other hand, Darcy system can still be solved successfully until $h\leq \frac{1}{1000}$ ( $H\leq\frac{1}{100}$ ). Meanwhile, \emph{CTG} may become unstable when solving finer grid problems, due to the failure of solving the coupled system at the first step. Besides, Table \ref{ctg}  demonstrate that \emph{CTG} can solve coarser mesh problems much fast than TGDDMs when  $H$ is bigger .  However,  \emph{TGDDM2} can save more CPU time compared with \emph{CTG} when $H$ is smaller.

Furthermore,  we demonstrate  the \textit{log-log} plots of the relatively approximate errors for velocity, pressure and hydraulic head  in Fig. \ref{p1b} for the both cases of $\delta_S=\delta_D$ and $\delta_S<\delta_D$ by using $P1b-P1-P1$ finite elements triple to validate the  Remark \ref{remark2}. From the figure, we can intuitively observe that the optimal convergence order  $O(h)$ is achieved for all these three quantities, which confirm the conclusion in Remark \ref{remark2}.

\begin{figure}[htbp]
\centering
\subfigure[ \emph{TGDDM1} ]{
		\includegraphics[width=60mm,height=60mm]{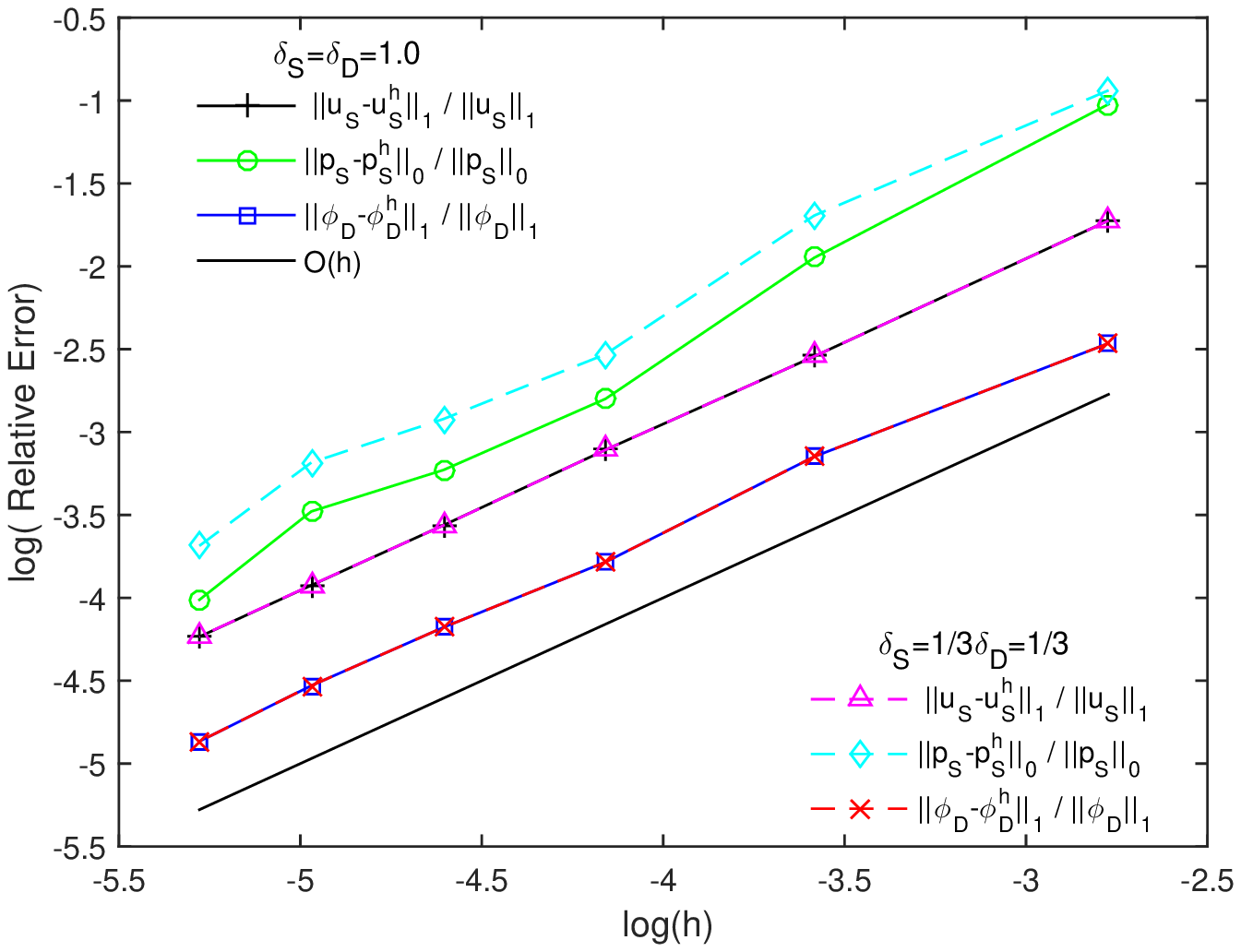}}
	\quad
	\subfigure[ \emph{TGDDM2}]{
		\includegraphics[width=60mm,height=60mm]{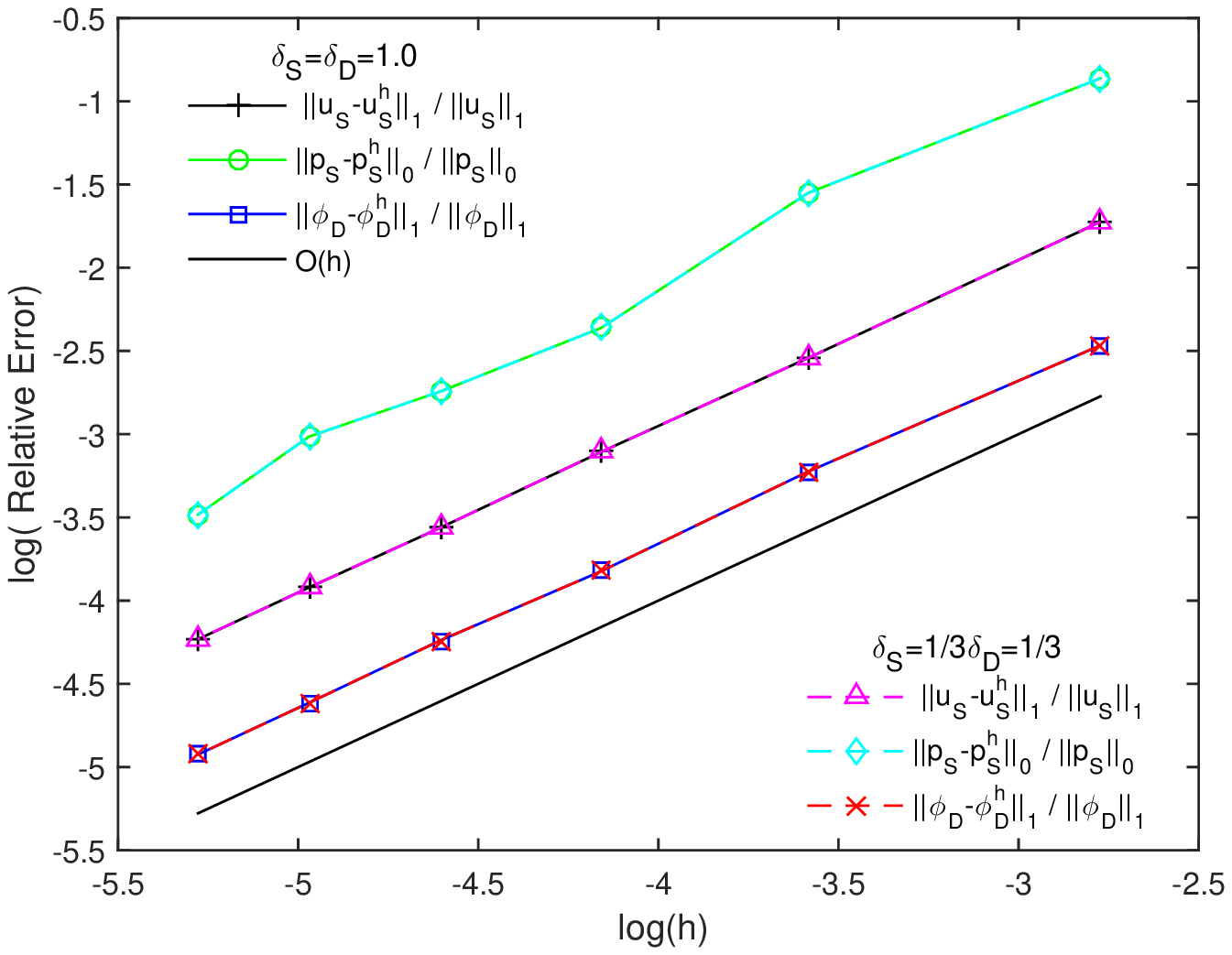}}
\caption{Error plots for \emph{TGDDMs} with different $\delta_S$ and $\delta_D$ by using $P1b-P1-P1$ finite elements.}\label{p1b}
\end{figure}

Finally, we compare the approximate accuracy of the proposed \emph{TGDDM1}  with  recently developed ``Two-level Optimized Schwarz method (\emph{TLOSM}): Algorithm $ 2.1 $ \cite{Gander}" in Table \ref{errosm} using $P1b-P1-P1$ finite elements triple. The \emph{TLOSM} algorithm  execute iteratively  several steps on the fine grid to solve the decoupling problems, and then computing the residual equations on the coarse grid to correct the  fine grid solutions. By following Gander et al. \cite{Gander}, we consider the coarse and fine meshes with relationship $h=H/2$, the fine grid iteration numbers $n_1=n_2=2$ and the optimized parameters $s_1=100, s_2=1/50$. On the other hand, we set $\delta_S=s_1$ and $\delta_D=s_2$ to simulate the proposed algorithm for the comparison purpose. As expected, both algorithms achieve optimal convergence  orders  but \emph{TGDDM1} can save  significant amount of CPU time \emph{TLOSM}. 

\begin{table}[!h]
\caption{The approximate accuracy and CPU performance of \emph{TGDDM1} algorithm and \emph{TLOSM} algorithm with $\delta_S=s_1=100$ and $\delta_D=s_2=1/50$ under $P1b-P1-P1$ finite elements triple.} \label{errosm}\tabcolsep
0pt \vspace*{-10pt}
\par
\begin{center}
\def\temptablewidth{1.0\textwidth}
{\rule{\temptablewidth}{1pt}}
\begin{tabular*}{\temptablewidth}{@{\extracolsep{\fill}}ccc||c|c|c|c||ccc||c|c|c|c}
\hline
\multirow{2}*{$ H $} & \multirow{2}*{$N$} & \multirow{2}*{$ h $} & \multicolumn{4}{c||}{\emph{TGDDM1}} & \multirow{2}*{$ H $} & \multirow{2}*{$ h $} & \multirow{2}*{$N$} & \multicolumn{4}{c}{\emph{TLOSM} ($h=H/2$)}\\
\cline{4-7}
\cline{11-14}
& & & $\frac{||\overrightarrow{u_S}-\overrightarrow{u_{S}^h}||_{1}}
{||\overrightarrow{u_S}||_{1}}$ \hspace{1mm} &$\frac{||p_S-p_{S}^h||}{||p_S||}$ \hspace{2mm} & $\frac{||\phi_D-\phi_D^h||_1}{||\phi_D||_1}$ \hspace{1mm} & CPU & & & & $\frac{||\overrightarrow{u_S}-\overrightarrow{u_{S}^h}||_{1}}
{||\overrightarrow{u_S}||_{1}}$ \hspace{1mm} & $\frac{||p_S-p_{S}^h||}{||p_S||}$ \hspace{2mm} & $\frac{||\phi_D-\phi_D^h||_1}{||\phi_D||_1}$ & CPU \\
\hline
$\frac{1}{4} $ & 9 & $\frac{1}{16}   $ & 0.178326 & 0.420358 & 0.084880 & 0.47 & $\frac{1}{8}  $ & $\frac{1}{16}  $ & 5 & 0.182883 & 0.662469 & 0.080170 & 1.60 \\
$\frac{1}{6} $ & 9 & $\frac{1}{36}   $  & 0.078727 & 0.210421 & 0.039834 & 1.94 & $\frac{1}{18} $ & $\frac{1}{36}  $ & 5 & 0.079088 & 0.310457 & 0.035791 & 7.26 \\
$\frac{1}{8} $ & 10 & $\frac{1}{64}   $  & 0.044738 & 0.093331 & 0.021860 & 4.04 & $\frac{1}{32} $ & $\frac{1}{64}  $ & 5 & 0.045216 & 0.129701 & 0.020445 & 21.26\\
$\frac{1}{10} $ & 10 & $\frac{1}{100} $  & 0.028379 & 0.063680 & 0.014426 & 8.38 & $\frac{1}{50} $ & $\frac{1}{100} $ & 5 & 0.028648 & 0.094091 & 0.013191 & 52.27\\
$\frac{1}{12} $ & 10 & $\frac{1}{144} $  & 0.019818 & 0.048571 & 0.009923 & 17.49 & $\frac{1}{72} $ & $\frac{1}{144} $ & 5 & 0.019830 & 0.052905 & 0.009002 & 112.81\\
$\frac{1}{14} $ & 10 & $\frac{1}{196} $  & 0.014515 & 0.030350 & 0.007270 & 29.70 & $\frac{1}{98} $ & $\frac{1}{196} $ & 5 & 0.014529 & 0.035051 & 0.006659 & 224.78\\
\hline
\multicolumn{3}{c||}{Rate} & 1.01 & 1.53 & 1.01 & & \multicolumn{3}{c||}{Rate} & 1.01 & 1.33 & 0.98 &  \\
\hline\hline
\end{tabular*}%
\end{center}
\end{table}

\section{Conclusions}
In this contribution, we propose two novel optimized two-grid domain decomposition algorithms to solve the coupled Stokes-Darcy multi-physics system numerically. We derive the stability and convergence analysis for the proposed TGDDMs algorithm rigorously. Moreover, several numerical experiments are preform to validate the proposed algorithms and demonstrate the accuracy and efficiency. The optimal convergence order is achieved for both schemes. The flow speed, streamlines, pressure contour, conservation of mass, and complicated flow characteristics are illustrated on the 2D and 3D geometrical setup for the Stokes-Darcy fluid flow model. All the numerical experiments indicate that the proposed TGDDMs algorithms consume very little CPU time and save computational resources than the classical two-grid methods. On the other hand, the newly optimized TGDDMs scheme can simulate large-scale multi-domain/multi-physics models for finer grids with more realistic physical parameters better than the DDMs or traditional two-grid methods. In the future, we can extend these newly optimized methods to solve other types of multi-domain, multi-physical models, such as the Stokes-Darcy system with Beavers-Joseph interface conditions, and the Navier-Stokes/Darcy system in the future.



\section{Acknowledgements}
The authors  would like to thank Gander M.J. and Vanzan T. to improve our numerical experiments.


\begin{thebibliography}{999}
\bibitem{LMLayton2003} { Layton WJ, Schieweck F, Yotov I:}  {Coupling fluid flow with porous media flow}. SIAM J. Numer. Anal., 40, 2195-2218 (2003).

\bibitem{LMGatica03} {Gatica GN, Oyarz\'{u}a R, Sayas FJ:}  {A conforming mixed finite-element method for the coupling of fluid flow with porous media flow}. IMA J. Num. Anal. 29, 86-108  (2009).

\bibitem{LMGatica2012} {Gatica GN, Oyarz\'{u}a R, Sayas FJ:}  {Analysis of fully-mixed finite element methods for the Stokes-Darcy coupled problem}. Math. Comput., 80, 1911-1948 (2011).

\bibitem{JP2020} {Yu JP, Sun YZ, Shi F, Zheng HB:} {Nitsche's type stabilized finite element method for the fully mixed Stokes-Darcy problem with Beavers-Joseph conditions}. Appl. Math. Letters 110, 106588 (2020).

\bibitem{Mahbub2019} {Mahbub MA A, He XM, Nasu NJ, Qiu CX, Zheng HB:} {Coupled and decoupled stabilized mixed finite element methods for nonstationary dual-porosity-Stokes fluid flow model}. Int. J. Numer. Methods Engin. 120,  803-833 (2019).

\bibitem{Mahbub2020} {Mahbub MA A, Shi F, Nasu NJ, Wang YS, Zheng HB:} {Mixed stabilized finite element method for the stationary Stokes-dual-permeability fluid flow model}. Comput. Meth. Appl. Mech. Engin. 358,  112616 (2020).

\bibitem{Discacciati07} {Discacciati M, Quarteroni A, Valli A:}  {Robin-Robin domain decomposition methods for the Stokes-Darcy coupling}.
 SIAM J. Numer. Anal., 45, 1246-1268 (2007).

\bibitem{Chen} {Chen W, Gunzburger M, Hua F, Wang X:}  {A parallel Robin-Robin domain decomposition method for the Stokes-Darcy system}. SIAM. J. Numer. Anal., 49, 1064-1084  (2011).


\bibitem{Cao14}{ Cao Y, Gunzburger M,  He XM, Wang X:} Parallel, non-iterative, multi-physics domain decomposition methods for time-dependent Stokes-Darcy systems. Math. Comput., 83, 1617-1644  (2014).

\bibitem{Boubendir}{Boubendir Y, Tlupova S:}  {Domain decomposition methods for solving Stokes-Darcy problems with boundary integrals}. SIAM J. Sci. Comput., 35, B82-B106  (2013).

\bibitem{Cao11} {Cao Y, Gunzburger M,  He XM, Wang X:} {Robin-Robin domain decomposition methods for the steady-state Stokes-Darcy system with Beaver-Joseph interface condition}. Numer. Math., 117, 601-629  (2011).



\bibitem{He15} {He XM, Li J, Lin YP, Ming J:}  {A domain decomposition method for the steady-state Navier-Stokes-Darcy model with the Beavers-Joseph interface condition}. SIAM J. Sci. Comput., 37, S264-S290 (2015).


\bibitem{Vassilev1} {Vassilev D, Wang C, Yotov I:}  {Domain decomposition for coupled Stokes and Darcy flows}. Comput. Methods Appl. Mech.
Engrg., 268, 264-283  (2014).

\bibitem{Jiang} {Jiang B:}  {A parallel domain decomposition method for coupling of surface and groundwater flows}. Comput. Methods Appl. Mech.
Engrg., 198, 947-957  (2009).

\bibitem{Mu} {Mu M, Xu J:}  {A two-grid method of a mixed Stokes-Darcy model for coupling fluid flow with porous media flow}.
 SIAM J. Numer. Anal., 45, 1801-1813 (2007).

\bibitem{Cai2009} {Cai M, Mu M, Xu J:}  {Numerical solution to a mixed Navier-Stokes/Darcy model by the two-grid approach}. SIAM J. Numer. Anal.,
47, 3325-3338 (2009).

\bibitem{Zuo14} {Zuo L, Hou Y:}  {A decoupling two-grid algorithm for the mixed Stokes-Darcy model with the Beavers-Joseph interface condition}.
 Numer. Methods PDEs,  30, 1066-1082  (2014).

 \bibitem{Hou16} {Hou Y:} {Optimal error estimates of a decoupled scheme based on two-grid finite element for mixed Stokes-Darcy model}. App. Math. Letters, 57, 90-96 (2016).

\bibitem{Zuo18} {Zuo L,   Du G:} {A multi-grid technique for coupling fluid flow with porous media flow}. Comput. Math. Appl.,  75(11), 4012-4021 (2018).

\bibitem{Discacciati} {Discacciati M, Gerardo-Giorda L:} {Optimized Scjwarz methods for the Stokes-Darcy coupling}. IMA J. Numer. Anal., 38, 1959-1983 (2018).

\bibitem{Gander} {Gander MJ, Vanzan T:} {Multilevel Optimized Schwarz Methods}. SIAM J. Sci. Comp., 42 (5), A3180-A3209 (2020).

\bibitem{Beavers} {Beavers G, Joseph D:}  {Boundary conditions at a naturally permeable wall}. J. Fluid Mech., 30, 197-207 (1967).

\bibitem{Saffman} {Saffman P:}  {On the boundary condition at the surface of a porous medium}. Stud. Appl. Math., 50, 93-101 (1971).

\bibitem{Jones}{Jones, IP:}  {Low Reynolds number flow past a porous spherical
shell}. Proc. Camb. Philol. Soc., 73, 231-238 (1973).

\bibitem{CGHW08} {Cao Y, Gunzburger M, Hu X, Hua F, Wang, X:}  {Coupled Stokes-Darcy Model with Beavers-Joseph Interface Boundary Condition}. Commun. Math. Sci., 8, 1-25  (2010).

\bibitem{Cao10} {Cao Y, Gunzburger M, Hu X, Hua F, Wang X, Zhao W:}  {Finite element approximations for Stokes-Darcy flow with Beavers-Joseph interface conditions}. SIAM. J. Numer. Anal., 47, 4239-4256  (2010).

\bibitem{Sobolev} { Adams RA, Fournier JJF:}  {Sobolev Spaces}. 2nd ed., Pure Appl. Math. (Amst.) 140, Elsevier/Academic Press, Amsterdam (2003).

\bibitem{proveLBB} {Girault V, Raviart P.-A:}  {Finite element methods for Navier-Stokes equations}. Apringer-Verlag, Berlin  (1986).

\bibitem{proveLBB2} {Gunzburger M:}  {Finite element methods for viscous incompressible flows: a guide to theory, practice, and algorithms}.  Academic Press, Boston (1989).

\bibitem{trace} {Thom$\acute{\mathrm{e}}$e V:}  {Galerkin finite element methods for parabolic problems}. Springer-Verlag, Berlin, second edition (2006).

\bibitem{F18}{Hecht F:} New development in freefem++. J. Numer. Math., 20, 251-265 (2012).

\bibitem{ex21} {Kanschat G, Rivi$\acute{e}$re B:}  {A strongly conservative finite element method for the coupling of Stokes and Darcy flow}. J. Comput. Phys., 229, 5933-5943 (2010).

\bibitem{ex22} {Li R, Li J, He XM, Chen ZX:}  {A stabilized finite volume element method for a coupled Stokes-Darcy problem}. Appl. Numer. Math., 133, 2-24 (2018).


\bibitem{3D} {Discacciati M, Miglio E, Quarteroni A:}  {Mathematical and numerical models for coupling surface and groundwater flows}. Appl. Numer. Math., 43, 57-74 (2002).




\end{thebibliography}
\end{document}